\documentclass[11pt]{article}
\usepackage{amsmath}
\usepackage{amsfonts}
\usepackage{amssymb, amsthm}
\usepackage{color}
\usepackage[margin=1in]{geometry}
\usepackage{hyperref}
\usepackage{verbatim}
\usepackage{ulem}
\usepackage{indentfirst}
\usepackage{graphicx}
\usepackage{caption}
\usepackage{subcaption}
\usepackage{multirow}

\newtheorem{theorem}[equation]{Theorem}
\newtheorem{proposition}[equation]{Proposition}
\newtheorem{lemma}[equation]{Lemma}
\newtheorem{corollary}[equation]{Corollary}
\newtheorem{conjecture}[equation]{Conjecture}

\theoremstyle{definition}
\newtheorem{defn}[equation]{Definition}
\newtheorem{definition}[equation]{Definition}
\newtheorem{eg}[equation]{Example}
\newtheorem{rmk}[equation]{Remark}
\newtheorem{remark}[equation]{Remark}

\newcommand{\epair}{(\emptyset; \emptyset)}

\newcommand{\C}{\mathbb{C}}

\newcommand{\mc}{\mathcal}

\allowdisplaybreaks

\title{Circular Planar Electrical Networks II: Positivity Phenomena}
\author{Joshua Alman, Carl Lian, and Brandon Tran\footnote{Department of Mathematics, Massachusetts Institute of Technology. \href{mailto:jalman@mit.edu}{jalman@mit.edu},  \href{mailto:clian@math.mit.edu}{clian@math.mit.edu}, \href{mailto:btran115@mit.edu}{btran115@mit.edu}}}

\begin{document}

\maketitle

\begin{abstract} 

Curtis-Ingerman-Morrow characterize response matrices for circular planar electrical networks as symmetric square matrices with row sums zero and non-negative circular minors. In this paper, we study this positivity phenomenon more closely, from both algebraic and combinatorial perspectives. Extending work of Postnikov, we introduce electrical positroids, which are the sets of circular minors which can simultaneously be positive in a response matrix. We give a self-contained axiomatic description of these electrical positroids. In the second part of the paper, we discuss a naturally arising example of a Laurent phenomenon algebra, as studied by Lam-Pylyavskyy. We investigate the clusters in this algebra, building off of initial work by Kenyon-Wilson, using an analogue of weak separation, as was originally introduced by Leclerc-Zelevinsky.

\end{abstract}

\tableofcontents

\section{Introduction}

\numberwithin{equation}{section}

Circular planar electrical networks were studied by Curtis-Ingerman-Morrow \cite{curtis} and de Verdi\'{e}re-Gitler-Vertigan \cite{french}. Associated to any circular planar electrical network of order $n$ is its $n\times n$ \textit{response matrix}, and response matrices are characterized in \cite[Theorem 4]{curtis} as the symmetric matrices with row sums equal to zero and \textit{circular minors} non-negative. Furthermore, the circular minors which are strictly positive can be identified combinatorially using \cite[Lemma 4.2]{curtis}. This positivity condition is of particular interest to us, and the goal of this paper is to investigate the connections between electrical networks and other positivity phenomena in the literature. Until now, the combinatorial properties of these response matrices have remained largely unstudied, despite their inherently combinatorial descriptions.

A natural question that arises is: which sets of circular minors can be positive, while the others are zero? It is clear (for example, from the Grassmann-Pl\"{u}cker relations) that one cannot construct response matrices with arbitrary sets of positive circular minors. Postnikov \cite{sasha} studied a similar question in the \textit{totally nonnegative Grassmannian}, as follows: for $k\times n$ matrices $A$, with $k<n$ and all $k\times k$ minors nonnegative, which sets (in fact, matroids) of $k\times k$ minors can be the set of positive minors of $A$? These sets, called \textit{positroids} by Knutson-Lam-Speyer \cite{knutson}, were found in \cite{sasha} to index many interesting combinatorial objects. Two of these objects, plabic graphs and alternating strand diagrams, are highly similar to circular planar electrial networks and medial graphs, respectively, which we study in this paper. Our introduction of \textit{electrical positroids} is therefore a natural extension of the theory of positroids. We give a novel axiomatization of electrical positroids, motivated by the Grassmann-Pl\"{u}cker relations, and prove the following:

\newtheorem*{introcircposthm}{Theorem \ref{circposthm}}
\begin{introcircposthm}
A set $S$ of circular pairs is the set of positive circular minors of a response matrix if and only if $S$ is an electrical positroid.
\end{introcircposthm}

Another point of interest is that of \textit{positivity tests} for response matrices. In \cite{totalpos}, Fomin-Zelevinsky describe various positivity tests for \textit{totally positive matrices}: given an $n\times n$ matrix, there exist sets of $n^{2}$ minors whose positivity implies the positivity of all minors. These sets of minors are described combinatorially by \textit{double wiring diagrams}. Fomin-Zelevinsky later introduced \textit{cluster algebras} in \cite{fz1}, in part, to study similar positivity phenomena. In particular, their double wiring diagrams fit naturally within the realm of cluster algebras as manifestations of certain cluster algebra mutations.

In a similar way, we describe sets of $\binom{n}{2}$ minors of an $n\times n$ matrix $M$ whose positivity implies the positivity of all circular minors, that is, that $M$ is a response matrix for a top-rank (in $EP_{n}$, as defined in \cite[\S 3]{paper1}) electrical network. Some such sets were first described by Kenyon-Wilson \cite[\S 4.5.3]{kenyon}. However, these sets do not form clusters in a cluster algebra. Instead, they form clusters in a \textit{Laurent phenomenon (LP) algebra}, a notion introduced by Lam-Pylyavskyy in \cite{lp}. This observation leads to the last of our main theorems:

\newtheorem*{introlpalg}{Theorem \ref{lpalg}}
\begin{introlpalg}
There exists an LP algebra $\mc{LM}_{n}$, isomorphic to the polynomial ring on $\binom{n}{2}$ generators, with an initial seed $\mc{D}_{n}$ of diametric circular minors. $\mc{D}_{n}$ is a positivity test for circular minors, and furthermore, all ``Pl\"{u}cker clusters'' in $\mc{LM}_{n}$, that is, clusters of circular minors, are positivity tests.
\end{introlpalg}

In proving Theorem \ref{lpalg}, we find that $\mc{LM}_{n}$ is, in a sense, ``double-covered'' by a cluster algebra $\mc{CM}_{n}$ that behaves very much like $\mc{LM}_{n}$ when we restrict to certain types of mutations. Further investigation of the clusters leads to an analogue of weak separation, as studied by Oh-Speyer-Postnikov \cite{osp} and Scott \cite{scott}. Conjecturally, the ``Pl\"{u}cker clusters,'' of $\mc{LM}_{n}$ correspond exactly to the maximal pairwise weakly separated sets of circular pairs. Furthermore, we conjecture that these maximal pairwise weakly separated sets are related to each other by mutations corresponding to the Grasmann-Pl\"{u}cker relations. While we establish several weak forms of the conjecture, the general statement remains open.

The roadmap of the paper is as follows. We begin by briefly reviewing terminology and known results in \S\ref{defs}, referring the reader to \cite[\S2]{paper1} for a more detailed exposition. In \S\ref{circpos}, we motivate and introduce electrical positroids, and prove Theorem \ref{circposthm}. In \S\ref{lpsec}, using positivity tests as a springboard, we construct $\mc{LM}_{n}$ and prove Theorem \ref{lpalg}, and conclude by establishing weak forms of Conjecture \ref{wsconj}, which relates the clusters of $\mc{LM}_{n}$ to positivity tests and our new analogue of weak separation.

\numberwithin{equation}{section}

\section{Preliminaries: Circular Planar Electrical Networks and Response Matrices}\label{defs}

We recommend that the reader begin by reviewing the definitions and notational conventions from \cite[\S 2]{paper1}. For convenience, we re-state some of the more important definitions and results that we will use in this paper.

\begin{defn} A \textbf{circular planar graph} $\Gamma$ is a planar graph embedded in a disk $D$. $\Gamma$ is allowed to have self-loops and multiple edges, and has at least one vertex on the boundary of $D$ - such vertices are called \textbf{boundary vertices}. A \textbf{circular planar electrical network} (or more simply an \textbf{electrical network} is a circular planar graph $\Gamma$, together with a \textbf{conductance map} $\gamma:E(\Gamma)\rightarrow\mathbb{R}_{>0}$. \end{defn}

Associated to an electrial network is its \textbf{response matrix}, which is obtained by imposing voltages at the boundary vertices and measuring the resulting boundary currents. More details are given in \cite[\S 2.1]{paper1}.

\begin{defn}
Two electrical networks $(\Gamma_{1},\gamma_{1}),(\Gamma_{2},\gamma_{2})$ are \textbf{equivalent} if they have the same response matrix. The equivalence relation is denoted by $\sim$.
\end{defn}

Recall also the \textbf{local equivalences}, described in detail in \cite{curtis} and \cite[\S 2.1]{paper1}: self-loop and spike removal, series and parallel edge replacement, and $Y$-$\Delta$ transformations. These local equivalences are sufficient to generate the equivalence relation $\sim$:

\begin{theorem}[{\cite[Th\'{e}or\`{e}me 4]{french}}]\label{localeqenough}
Two electrical networks are equivalent if and only if they are related by a sequence of local equivalences.
\end{theorem}

As in \cite{paper1}, we will often consider only the underlying circular planar graph of an electrical network. Accordingly, we use the following notion of equivalence of circular planar graphs:

\begin{defn}
Let $\Gamma_{1},\Gamma_{2}$ be circular planar graphs, each with the same number of boundary vertices. Then, $\Gamma_{1},\Gamma_{2}$ are \textbf{equivalent} (with the equivalence abusively denoted $\sim$) if there exist conductances $\gamma_{1},\gamma_{2}$ on $\Gamma_{1},\Gamma_{2}$, respectively such that $(\Gamma_{1},\gamma_{1}),(\Gamma_{2},\gamma_{2})$ are equivalent electrical networks. Alternatively, $\Gamma_{1}\sim\Gamma_{2}$ if $\Gamma_{1}$ and $\Gamma_{2}$ are related by a sequence of local equivalences.
\end{defn}

The central ingredient to this paper is the characterization of response matrices given in \cite{curtis}. In order to state this characterization, we recall the definitions of circular pairs and circular minors.

\begin{defn}\label{circpair}
Let $P=\{p_{1},p_{2},\ldots,p_{k}\}$ and $Q=\{q_{1},q_{2},\ldots,q_{k}\}$ be disjoint ordered subsets of the boundary vertices of an electrical network $(\Gamma,\gamma)$. We say that $(P;Q)$ is a \textbf{circular pair} if $p_{1},\ldots,p_{k},q_{k},\ldots,q_{1}$ are in clockwise order around the circle. We will refer to $k$ as the \textbf{size} of the circular pair.
\end{defn}

\begin{rmk}\label{flipPQ}
We will take $(P;Q)$ to be the same circular pair as $(\widetilde{Q};\widetilde{P})$, where $\widetilde{P}$ denotes the ordered set $P$ with its elements reversed. Almost all of our definitions and statements are compatible with this convention; most notably, by Theorem \ref{circularminorspositive}(a), because response matrices are positive, the circular minors $M(P;Q)$ and $M(\widetilde{Q};\widetilde{P})$ are the same. Whenever there is a question as to the effect of choosing either $(P;Q)$ or $(\widetilde{Q};\widetilde{P})$, we take extra care to point the possible ambiguity.
\end{rmk}

\begin{defn}\label{connection}
Let $(P;Q)$ and $(\Gamma,\gamma)$ be as in Definition \ref{circpair}. We say that there is a \textbf{connection} from $P$ to $Q$ in $\Gamma$ if there exists a collection of vertex-disjoint paths from $p_{i}$ to $q_{i}$ in $\Gamma$, and furthermore each path in the collection contains no boundary vertices other than its endpoints. We denote the set of circular pairs $(P;Q)$ for which $P$ is connected to $Q$ by $\pi(\Gamma)$.
\end{defn}

\begin{defn}
Let $(P;Q)$ and $(\Gamma,\gamma)$ be as in Definition \ref{circpair}, and let $M$ be the response matrix. We define the \textbf{circular minor} associated to $(P;Q)$ to be the determinant of the $k\times k$ matrix $M(P;Q)$ with $M(P;Q)_{i,j}=M_{p_{i},q_{j}}$.
\end{defn}

As in \cite{paper1}, we often refer to submatrices and their determinants both as minors, interchangeably.

We now have the language needed to state the characterization of response matrices:

\begin{theorem}[{\cite[Theorem 2.2.6]{curtis}}]\label{circularminorspositive}
Let $M$ be an $n\times n$ matrix. Then:
\begin{enumerate}
\item[(a)] $M$ is the response matrix for an electrical network $(\Gamma,\gamma)$ if and only if $M$ is symmetric with row and column sums equal zero, and each of the circular minors $M(P;Q)$ is non-negative. 
\item[(b)] If $M$ is the response matrix for an electrical network $(\Gamma,\gamma)$, the positive circular minors $M(P;Q)$ are exactly those for which there is a connection from $P$ to $Q$.
\end{enumerate}
\end{theorem}

Let us also mention include two more tools which we will need in what follows.

\begin{theorem}\label{equivalentconnections}
The circular planar graphs $G_1$ and $G_2$ are equivalent if and only if $\pi(G_{1})=\pi(G_{2})$.
\end{theorem}
\begin{proof}
It is easily checked that local equivalences preserve $\pi$, so if $G_{1}\sim G_{2}$, then $\pi(G_{1})=\pi(G_{2})$. Conversely, if $\pi(G_{1})=\pi(G_{2})$, there exist critical (see \cite[\S 3]{paper1} graphs $G'_{1},G'_{2}$ with $G_{1}\sim G'_{1}$ and $G_{2}\sim G'_{2}$, by  \cite[Th\'{e}or\`{e}me 2]{french}. Then, $G'_{1}\sim G'_{2}$, so the claim follows from \cite[Theorem 1]{curtis}.
\end{proof}

\begin{defn}
Let $G$ be a circular planar graph, and let $e$ be an edge with endpoints $v,w$. The \textbf{deletion} of $e$ from $G$ is exactly as named; the edge $e$ is removed while leaving the rest of the vertices and edges of $G$ unchanged. If $v,w$ are not both boundary vertex of $G$, we may also perform a \textbf{contraction} of $e$, which identifies all points of $e$. If exactly one of $v,w$ is a boundary vertex, then the image of $e$ under the contraction is a boundary vertex. Note that edges connecting two boundary vertices cannot be contracted to either endpoint.
\end{defn}

\numberwithin{equation}{subsection}

\section{Electrical Positroids}\label{circpos}

By Theorem \ref{circularminorspositive}, $n\times n$ response matrices are characterized in the following way: a square matrix $M$ is the response matrix for an electrical network $(\Gamma,\gamma)$ if and only if $M$ is symmetric, its row and column sums are zero, and its circular minors $M(P;Q)$ are non-negative. Furthermore, $M(P;Q)$ is positive if and only if there is a connection from $P$ to $Q$ in $\Gamma$. The sets $S$ of circular pairs for which there exists a response matrix $M$ with $M(P;Q)$ is positive if and only if $(P;Q)\in S$, then, are thus our next objects of study.

The case of the totally nonnegative Grassmannian was studied in \cite{sasha}: for $k\times n$ (with $k<n$) matrices with non-negative maximal minors, the possible sets of positive maximal minors are called \textit{positroids}, and are a special class of matroids. Our objects will be called \textit{electrical positroids}, which we first construct axiomatically, then prove are exactly those sets $S$ of positive circular minors in response matrices.

\subsection{Grassmann-Pl\"{u}cker Relations and Electrical Positroid Axioms}

Here, we present the axioms for electrical positroids, which arise naturally from the Grassmann-Pl\"{u}cker Relations.

\begin{defn}
Let $M$ be a fixed matrix, whose rows and columns are indexed by some sets $I,J$. We write $\Delta^{i_{1}i_{2}\cdots i_{m},j_{1}j_{2}\cdots j_{n}}$ for the determinant of the matrix $M'$ formed by deleting the rows corresponding to $i_{1},i_{2},\ldots,i_{m}\in I$ and $j_{1},j_{2},\ldots,j_{n}\in J$, provided $M'$ is square.
\end{defn}

While the meaning $\Delta^{i_{1}i_{2}\cdots i_{m},j_{1}j_{2}\cdots j_{n}}$ depends on the underlying sets $I,J$, these sets will always be implicit.

\begin{proposition}\label{gp}
We have the following two \textbf{Grassmann-Pl\"{u}cker} relations.
\begin{enumerate}
\item[(a)] Let $M$ be an $n\times n$ matrix, with $a,b$ elements of its row set and $c,d$ elements of its column set. Furthermore, suppose that the row $a$ appears above  row $b$ and column $c$ appears to the left of column $d$. Then,
\begin{align}\label{P1} 
\Delta^{a,c}\Delta^{b,d}=\Delta^{a,d}\Delta^{b,c}+\Delta^{ab,cd}\Delta^{\emptyset,\emptyset}.
\end{align}
\item[(b)] Let $M$ be an $(n+1)\times n$ matrix, with $a,b,c$ elements of its row set (appearing in this order, from top to bottom), and let $d$ an element of its column set. Then,
\begin{align}\label{P2}
\Delta^{b,\emptyset}\Delta^{ac,d}=\Delta^{a,\emptyset}\Delta^{bc,d}+\Delta^{c,\emptyset}\Delta^{ab,d}.
\end{align}
\end{enumerate}
\end{proposition}

While the Grasmann-Pl\"{u}cker relations are purely algebraic in formulation, they encode combinatorial information concerning the connections of circular pairs in a circular planar graph $\Gamma$. As a simple example, consider four boundary vertices $a,b,d,c$ in clockwise order of an electrical network $(\Gamma,\gamma)$, and let $\pi=\pi(\Gamma)$. If $M$ is the response matrix of $(\Gamma,\gamma)$, then $M'=M(\{a,b\},\{c,d\})$ is the circular minor associated to the circular pair $(a,b;c,d)$; thus, $M'$ has non-negative determinant. Furthermore, the entries of $M'$ are $1\times1$ circular minors of $M$, so they, too, must be non-negative.

Now, suppose that the left hand side of (\ref{P1}) is positive, that is, $\Delta^{a,c}\Delta^{b,d}>0$. Equivalently, there are connections between $b$ and $d$ and between $a$ and $c$ in $\Gamma$. Then, at least one of the two terms on the right hand side must be strictly positive; combinatorially, this means that either there are connections between $b$ and $c$ and between $a$ and $d$, or there is a connection between $\{a,b\}$ and $\{c,d\}$. One can derive similar combinatorial rules by assuming one of the terms on the right hand side is positive, and deducing that the left hand side must be positive as well.

The first six of the electrical positroid axioms given in Definition \ref{posaxioms} summarize all of the information that can be extracted in this way from the Grassmann-Pl\"{u}cker relations.

\begin{defn}
If $a\in P$, write $P-a$ for the ordered set formed by removing $a$ from $P$.
\end{defn}

\begin{defn}\label{posaxioms}
A set $S$ of circular pairs is an \textbf{electrical positroid} if it satisfies the following eight axioms:
\begin{enumerate}
\item For ordered sets $P=\{a_1,a_2,\ldots,a_N\}$ and $Q=\{b_1,b_2,\ldots,b_N\}$, with $a_{1},\ldots,a_{N},b_{N},\ldots,b_{1}$ in clockwise order (that is, $(P;Q)$ is a circular pair), consider any $a=a_i$, $b=a_j$, $c=b_k$, $d=b_\ell$ with $i<j$ and $k<\ell$. Then:
\begin{enumerate}
\item \label{axside} If $(P-a; Q-c),(P-b; Q-d)\in S$, then either $(P-a;Q-d),(P-b;Q-c)\in S$ or $(P-a-b;Q-c-d),(P;Q)\in S$.
\item \label{axcross} If $(P-a;Q-d),(P-b;Q-c)\in S$, then $(P-a; Q-c),(P-b; Q-d)\in S$.
\item \label{axsep} If $(P-a-b;Q-c-d),(P;Q)\in S$, then $(P-a; Q-c),(P-b; Q-d)\in S$.
\end{enumerate}
\end{enumerate}
\begin{enumerate}
\setcounter{enumi}{1}
\item For $P=\{a_1,a_2,\ldots,a_{N+1}\}$ and $Q=\{b_1,b_2,\ldots,b_N\}$, with $a_{1},a_{2},\ldots,a_{N+1},b_{N},\ldots,b_{1}$ in clockwise order, consider any $a=a_i$, $b=a_j$, $c=a_k$, $d=b_\ell$ with $i<j<k$. Then:
\begin{enumerate}
\item \label{axmid} If $(P-b;Q),(P-a-c;Q-d)\in S$, then either $(P-a;Q),(P-b-c;Q-d)\in S$ or $(P-c;Q),(P-a-b;Q-d)\in S$.
\item \label{axleft} If $(P-a;Q),(P-b-c;Q-d)\in S$, then $(P-b;Q),(P-a-c;Q-d)\in S$.
\item \label{axright} If $(P-c;Q),(P-a-c;Q-d)\in S$, then $(P-b;Q),(P-a-c;Q-d)\in S$.
\end{enumerate}
\end{enumerate}
Finally:
\begin{enumerate}
\setcounter{enumi}{2}
\item \label{axsub} (\textbf{Subset axiom}) For $P=\{a_1,a_2,\ldots,a_n\}$ and $Q=\{b_1,b_2,\ldots,b_n\}$ with $(P;Q)$ a circular pair, if $(P;Q)\in S$, then $(P-a_i;Q-b_i)\in S$.
\item \label{axempty} $\epair\in S$.
\end{enumerate}
\end{defn}

\begin{theorem}\label{circposthm}
A set $S$ of circular pairs is an electrical positroid if and only if there exists a response matrix whose positive circular minors are exactly those corresponding to $S$.
\end{theorem}

Given a response matrix $M$, it is straightforward to check that the set $S$ of circular pairs corresponding to the positive circular minors of $M$ satisfies the first six axioms, by Proposition \ref{gp}. $S$ also satisfies the Subset Axiom, by Theorem \ref{circularminorspositive}. Finally, adopting the convention that the empty determinant is equal to 1, we have the last axiom. To prove Theorem \ref{circposthm}, we thus need to show that any electrical positroid $S$ may be realized as the set of positive circular minors of a response matrix, or equivalently the set of connections in a circular planar graph.

\subsection{Proof of Theorem \ref{circposthm}}\label{posproof}

We now prove Theorem \ref{circposthm}. First, recall the important convention that $(P;Q)=(\widetilde{Q};\widetilde{P})$. We leave it to the reader to check, whenever appropriate, that all of the definitions and statements we make in this section are compatible with this convention.

Fix a boundary circle with $n$ boundary vertices, which we label $1,2,\ldots,n$ in clockwise order. In this section, all labels are considered modulo $n$. We have shown, via the Grassmann-Pl\"{u}cker Relations, that the set of circular pairs corresponding to the positive circular minors of a response matrix is an electrical positroid. We now prove that, for all electrical positroids $S$, there exists a critical graph $G$ for which $\pi(G)=S$, which will establish Theorem \ref{circposthm}. The idea of the argument is as follows.

Assume, for sake of contradiction, that there exists some electrical positroid $S$ for which there does not exist such a critical graph $G$ with $\pi(G)=S$. Then, let $S_{0}$ have maximal size among all such electrical positroids. Note that $S_{0}$ does not contain all circular pairs $(P;Q)$, because otherwise $S_{0}=\pi(G_{\text{max}})$, where $G_{\text{max}}$ denotes a critical representative of the top-rank element of $EP_{n}$ (see \cite[\S 3]{paper1}).

We will then add circular pairs to $S_{0}$ according to the boundary edge and boundary spike properties (cf. \cite[\S 4]{curtis}), discussed below, to form an electrical positroid $S_{1}$. By the maximality of $S_{0}$, $S_{1}=\pi(G_{1})$ for some critical graph $G_{1}$. We will then delete a boundary edge or contract a boundary spike in $G_{1}$ to obtain a graph $G_{0}$, and show that $\pi(G_{0})=S_{0}$.

We begin by defining two properties of circular pairs, the $(i,i+1)$-boundary edge property and the $i$-boundary spike property. Let us first adopt a notational convention.

\begin{defn}
Given a circular pair $(P;Q)$, let $(P+x;Q+y)$ denote the unique circular pair (if it exists) with $P+x=P\cup\{x\}$ and $Q+y=Q\cup\{y\}$ as sets. In the ordered sets $P+x,Q+y$, $x,y$ are inserted in the appropriate positions so that $(P+x;Q+y)$ is indeed a circular pair.
\end{defn}

Given arbitrary $P,Q,x,y$, $(P+x;Q+y)$ may not be a circular pair. However, whenever we make reference to a pair of this form without commenting on its existence, we assert implicitly that it is, in fact, a circular pair.

\begin{defn}
We say that a set $S$ of circular pairs has the \textbf{$(i,i+1)$-BEP (boundary edge property)} if, for all circular pairs $(P;Q)\in S$, if $(P+i;Q+(i+1))$ is a circular pair, then $(P+i;Q+(i+1))\in S$.
\end{defn}

\begin{rmk}
According to Remark \ref{flipPQ}, if $S$ has the $(i,i+1)$-BEP, then if $(P;Q)\in S$ and $(P+(i+1);Q+i)$ is a circular pair, then $(P+(i+1);Q+i)\in S$.
\end{rmk}

\defn{We say that a set $S$ of circular pairs has the \textbf{$i$-BSP (boundary spike property)} if, for any circular pairs $(P;Q)\in S$ and $x,y$ such that $(P+x;Q+i),(P+i;Q+y)\in S$, we have $(P+x;Q+y)\in S$.}

\begin{lemma}\label{graphBEPBSP}
Recall the definitions of boundary edges and boundary spikes from \cite[\S 4]{curtis}. Let $G$ be a circular planar graph, and write $S=\pi(G)$.
\begin{enumerate}
\item[(a)] There exists $H\sim G$ with a boundary edge $(i,i+1)$ if and only if $S$ has the $(i,i+1)$-BEP.
\item[(b)] There exists $H\sim G$ has a boundary spike at $i$ if and only if $S$ has the $i$-BSP.
\end{enumerate}
\end{lemma}

\begin{proof}
We prove (a); the proof of (b) is similar. Without loss of generality, we may assume that $i=1$. It easy to check that if $G$ has a boundary edge, then $S$ must have the corresponding BEP. Conversely, suppose that $S$ has the $(1,2)$-BEP. Then, let $G'$ be the graph obtained by adding an edge $(1,2)$ in $G$ such that the added edge does not cut through any faces of $G$. Clearly, $\pi(G')-S$ consists only of circular pairs $(P;Q)$ such that $(P+1;Q+2)\in S$ or $(P+2;Q+1)\in S$. However, $S$ contains all such circular pairs, so in fact $\pi(G')=\pi(G)$. Then, by Theorem \ref{equivalentconnections}, $G'\sim G$.

It is left to check that $G\sim H$, for some circular planar graph $H$ with the boundary edge $(1,2)$. Let $a,b$ be the two medial boundary vertices between $1$ and $2$ in $\mc{M}(G)$. Note that adding the edge $(1,2)$ to $G$ corresponds to introducing an additional crossing in $\mc{M}(G)$ between the (distinct) wires with endpoints $a$ and $b$. Introducing this new crossing yields an equivalent medial graph, so it must have created it a lens. From here, it is easily seen, after applying \cite[Lemma 6.3]{curtis}, that motions may be applied in $\mc{M}(G')$ so that this lens corresponds to parallel edges between the boundary vertices $1$ and $2$ in some $H\sim G'\sim G$. The desired conclusion follows.
\end{proof}

\begin{lemma}\label{allBEPBSP}
If $S$ has all $n$ BEPs and all $n$ BSPs, then $S$ contains all circular pairs.
\end{lemma}

\begin{proof}
We proceed by induction on the size of $(P;Q)$ that $(P;Q)\in S$ for all circular pairs $(P;Q)$. First, suppose that $|P|=1$. First, $(i;i+1)\in S$ for all $i$, because it has all BEPs and $\epair\in S$. Then, because $S$ has the $i$-BSP, and $(i-1;i),(i;i+1)\in S$ we obtain $(i-1;i+1)\in S$. Continuing in this way gives that $S$ contains all circular pairs $(P;Q)$ with $|P|=1$. 

Now, suppose that $S$ contains all circular pairs of size $k-1$. Let $(a_1,\ldots a_k;b_1,\ldots b_k)$ be a circular pair of size $k$. By assumption, $(a_2,\ldots a_k; b_2,\ldots b_k)\in S$. Because $S$ has all BEPs, $(b_1+1, a_2,\ldots a_k; b_1, b_2,\ldots b_k)\in S$ and $(b_1+2,a_{2}\ldots a_k; b_1+1,b_2,\ldots b_k)\in S$, so by the $(b_{1}+1)$-BSP, $(b_1+2,a_2,\ldots a_k; b_1, b_2,\ldots b_k)\in S$. Continuing in this way gives $(a_1,\ldots a_k; b_1,\ldots b_k)\in S$, so we have the desired claim.
\end{proof}

In particular, Lemma \ref{allBEPBSP} tells us that there exists an $i$ such that $S_{0}$, as defined in the beginning of this section, either does not have the $(i,i+1)$-BEP for some $i$, or does not have the $i$-BSP for some $i$. We first assume that $S_{0}$ does not have all BEPs; without loss of generality, suppose that $S_{0}$ does not have the $(n,1)$-BEP.

We will now add circular pairs to $S_{0}$ to obtain an electrical positroid $S_{1}$ that does have the $(n,1)$-BEP. Specifically, we add to $S_{0}$ every circular pair $(P+1;Q+n)$, where $(P;Q)\in S_{0}$ has $1<a_{1}<b_{1}<n$ (here $P=\{a_{1},\ldots,a_{k}\},Q=\{b_{1},\ldots,B_{k}\}$), to obtain $S_{1}$. According to Remark \ref{flipPQ}, this construction also puts any $(P+n;Q+1)\in S_{1}$, where $(P;Q)\in S_{0}$ and $1<b_{k}<a_{k}<n$.

\begin{lemma}
$S'$ is an electrical positroid, and has the $(n,1)$-BEP.
\end{lemma}
\begin{proof}
The proof is straightforward, so it is omitted.
\end{proof}

By assumption, $S_{0}$ is the maximal electrical positroid for which any circular planar graph $G$ has $\pi(G)\neq S_{0}$. Thus, there exists a graph $G_{1}$ be a graph such that $\pi(G_{1})=S_{1}$, and $G_{1}$ may be taken to have a boundary edge $(n,1)$ by Lemma \ref{graphBEPBSP}. Then, let $G_{0}$ be the result of deleting the boundary edge $(n,1)$. To obtain a contradiction, it is enough to prove that $S_{0}=\pi(G_{0})$.

We now present a series of technical lemmas.

\begin{defn}\label{completedef}
Consider a circular pair $(P;Q)\in S_{0}$ for which $1,n\notin P\cup Q$. We will assume, for the rest of this section, that $(P+1;Q+n)$ is a circular pair. $(P;Q)$ is said to be is \textbf{incomplete} if $(P+1;Q+n)\notin S_{0}$, and \textbf{complete} if $(P+1; Q+n)\in S_{0}$.
\end{defn}

\begin{lemma}\label{crossing}
Let $(P;Q)=(a_{1},\ldots,a_{k};b_{1},\ldots,b_{k})\in S_{0}$ be an incomplete circular pair, such that $(P+1;Q+n)$ is a circular pair (and is not in $S_{0}$). Furthermore, assume that $(P;Q)$ is \textbf{minimal}, that is, $(P-a_{k};Q-b_{k})$ is complete. Then, for all $0\le i\le k-1$, $(a_i;b_{i+1}),(a_{i+1};b_i)\in S_{0}$. 
\end{lemma}

\begin{proof}
Immediate from Axiom \ref{axside} of Definition \ref{posaxioms}.
\end{proof}

\begin{lemma}\label{trip}
Let $(a,b,c;d,e,f)$ be a circular pair. Then, if $(a;d),(a;f),(b;e),(c;d),(c;f)\in S_{0}$, then $(a;d),(b;e),(b;f),(c;e)\in S_{0}$.
\end{lemma}

\begin{proof}
Immediate from Axiom \ref{axcross}.
\end{proof}

\begin{lemma}\label{walk}
If $(a_1,\ldots a_n; b_1, \ldots b_n)\in S_{0}$, $(a_{n+1};b_{n+1})\in S_{0}$, and $a_{n},a_{n+1},b_{n+1},b_{n}$ appear in clockwise order, then $(a_1,\ldots,a_{n-1},a_{n+1}; b_1, \ldots,b_{n-1}, b_{n+1})\in S_{0}$.
\end{lemma}

\begin{proof}
If $(a_n;b_{n+1})\in S$ and $(a_{n+1};b_n)\in S$, the claim follows from Axiom \ref{axleft}, Axiom \ref{axright} and induction on $n$. Otherwise, it follows from Axiom \ref{axside} and induction on $n$.
\end{proof}

\begin{lemma}\label{min}
Let $(P;Q)=(a_{1},\ldots,a_{k};b_{1},\ldots,b_{k})\in S_{0}$ be a complete circular pair. Then, $(P-a_{i};Q-b_{i})$ is complete for all $i=1,2,\ldots,k$.
\end{lemma}
\begin{proof}
Applying Axiom \ref{axright} with $a_1,a_i,a_k,b_1$ to $(P;Q-b_k)$ gives $(P-a_i;Q-b_k)\in S$. Then another application of Axiom \ref{axright}, to $(Q;P-a_i)$ with $b_1,b_i,b_k,a_i$ gives the desired result.
\end{proof}

\begin{lemma}\label{rtol}
Let $(P,a,b,c,Q;R,d,e,f,T)$ be a circular pair, where $P,Q,R,T$ are sequences of boundary vertices. Suppose that
\begin{align*}
(a;d),(a;e),(b;d),(b;e),(b;f),(c;e),(c;f)&\in S, \\
(P,a,b;R,d,e)&\in S,\text{ and}\\
(P,a,c,Q;R,d,f,T)&\in S
\end{align*}
Then, $(P,a,b,Q;R,d,e,T)\in S$.
\end{lemma}

\begin{proof}
First, write $P=P'\cup\{p\}, R=R'\cup\{r)\}$, where $p$ and $r$ are the last elements of the ordered sets $P,R$, respectively. Then, if $(P',b;R,f)\in S$, an application of Axiom \ref{axleft} on $(f,e,p,P';b,r,R')$ with $f,e,r,p$ yields $(P,b;R,f)\in S$. Similarly, we find by induction that $(b;f)\in S\Rightarrow (P,b;R,f)\in S$. Then, we have $(P,a,b;R,d,f)\in S$ by Axiom \ref{axleft} applied to $(f,e,d,R;b,a,P)$ with $f,e,d,a$. Similarly, write $Q=\{q\}\cup Q', T=\{t\}\cup T'$, where $q,t$ are the first elements of $Q,T$, respectively. By Axiom \ref{axleft} applied to $(P,a,b,c,q,Q';R,d,f,t,T')$ with $b,c,q,t$, we see that $(P,a,b,Q;R,d,f,T)\in S$. The lemma then follows from Axiom \ref{axright} applied to $(T,f,e,d,R;Q,b,a,P)$ with $T,f,e,Q$.
\end{proof}

\begin{lemma}\label{ltol}
Let $P,Q,R,T$ be sequences of indices, and let
$(1,P,a,b,c,Q;n,R,d,e,f,T)$
be a circular pair. Suppose 
\begin{align*}
(a;d),(a;e),(b;d),(b;e),(b;f),(c;e),(c;f)&\in S,\\
(1,P,a,b,Q;n,R,d,e,T)&\in S, \text{ and}\\
(P,a,c,Q;R,d,f,T)\in S. 
\end{align*}
Then $(1,P,a,c,Q;n,R,d,f,T)\in S$.
\end{lemma}
\begin{proof}
With the same notation as in the previous lemma, $(a,c,Q';d,e,T')\in S\Rightarrow (a,c,Q;d,e,T)\in S$ by Axiom \ref{axright} on $(T',t,f,e,d;Q',q,c,a)$ with $t,f,e,q$. Then, an inductive argument shows that $(a,c,Q;d,e,T)\in S$. A similar argument shows that $(P,a,c,Q;R,d,e,T)\in S$. Then, Axiom \ref{axright} applied to $(1,P,a,b,c,Q;n,R,d,e,T)$ with $1,b,c,n$ implies that $(1,P,a,c,Q;n,R,d,e,T)\in S$ and applying Axiom \ref{axright} again to $(1,P,a,c,Q;n,R,d,e,f,T)$ with $n,e,f,1$ yields the desired result.
\end{proof}

\begin{lemma}\label{Pa1Qbn}
Consider a circular pair $(P;Q)=(a_{1},\ldots,a_{k};b_{1},\ldots,b_{k}\}$, and let $(P+a;Q+b)$ be an incomplete circular pair with $a_{k}<a<b<b_{k}$ in clockwise order. Then, any electrical positroid $Z$ satisfying $S_{0}\cup\{(P+1;Q+n)\}\subset Z\subset S_{1}$ contains $(P+a+1;Q+b+n)$.
\end{lemma}
\begin{proof}
It is easy to see that any element of $Z\setminus S_{0}$ must be of the form $(P'+1; Q'+n)$, for some $P',Q'$. By Axiom \ref{axside}, $(P+1;Q+n)\in Z$ and $(P+a;Q+b)\in Z$ implies that either $(P+a+1;Q+b+n)\in Z$, or $(P+a+1;Q+b+n)\notin Z$ and $(P+1;Q+b),(P+a,Q+2)\in Z$. We are done in the former case, so assume for sake of contradiction that we have the latter. $(P+1;Q+b),(P+a,Q+2)$ are not of the form $(P'+1;Q'+n)$, so cannot lie in $Z\setminus S$; thus, $(P+1;Q+b),(P+a,Q+2)\in S$. Finally, Axiom \ref{axcross} yields us $(P+1;Q+n)\in S$, a contradiction, so we are done.
\end{proof}

\begin{defn}
Two pairs of indices $(i,j)$ and $(i',j')$ are said to \textbf{cross} if $i<i'<j'<j$ and $(i;j'),(i';j)\in S$.
\end{defn}

\begin{defn}\label{betternotation}
For ease of notation, denote the sequence of indices $a_{k},\ldots,a_{\ell}$ by $A_{k,\ell}$.
\end{defn}

We now algorithmically construct a set $\mc{P}$ of circular pairs, which we will call \textbf{primary} circular pairs. We will use this notion to eventually prove Lemma \ref{uniclos}, a key ingredient in our proof of the main theorem. The construction is as follows: begin by placing $(1;n)\in\mc{P}$. Then, for each $(P;Q)=(A_{1,i-1}; B_{1,i-1})\in\mc{P}$, if we also have $(P;Q)\in S_{0}$, perform the following operation.

\begin{itemize}
\item Let $a$ be the first index appearing clockwise from $a_{i-1}$ such that there exists $c$ with $(a,c)$ crossing $(a_{i-1},b_{i-1})$, and also $(A_{2,i-1}, a; B_{2,i-1}, c)\in S$. If $a$ does not exist, stop. Otherwise, with $a$ fixed, take $c$ to be the first index appearing counterclockwise from $b_{i-1}$ satisfying these properties.

\item If $a$ exists, add $(A_{1,i-1}, a; B_{1,i-1},c)$ to $\mc{P}$, and remove $(P;Q)=(A_{1,i-1};B_{1,i-1})$.

\item Similarly, let $b$ to be the largest index counterclockwise from $b_{i-1}$ such that there exists $d$ with $(d,b)$ crossing $(a_{i-1},b_{i-1})$ and $(a_2,\ldots d; b_2,\ldots b_i)\in S$. If $b$ does not exist, stop. Otherwise, with $b$ fixed, take $d$ to be the first index clockwise from $a_{i-1}$ with these properties. 

\item If $a\neq d$ and $b\neq c$ (note that if $a=d$, then $b=c$), then add $(A_{1,i-1}, d; B_{1,i-1},b)$ to $\mc{P}$. Note that $c\le d$ or else, by \ref{axside}, $c$ could originally have been set to $d$.
\end{itemize}

It is easily seen that at any time, the algorithm may be performed on the elements of $\mc{P}$ in any order, and that it will eventually terminate, when the operation described above results in no change in $\mc{P}$ for all $(P;Q)\in\mc{P}$.

\begin{defn}\label{connset}
For a circular pair $(P;Q) = (p_1, \ldots, p_k; q_1, \ldots, q_k)$, define $E(P;Q) = \{ \{p_i, q_i\} \mid i \in \{1, \ldots, k \}\}$. We will take $E(P;Q)$ to be an ordered set and abusively refer to its elements as \textbf{connections}.
\end{defn}

\begin{lemma}\label{PCPcont}
For any incomplete circular pair $(P;Q)$, there exists a circular pair $(P';Q')\in \mc{P}$ such that any electrical positroid $Z$ satisfying $S_{0}\cup \{(P';Q')\}\subset Z\subset S_{1}$ contains $(P+1;Q+n)$.
\end{lemma}

\begin{proof}
By Lemma \ref{Pa1Qbn}, we may assume that $(P;Q)$ is a minimal incomplete circular pair. Let $(P+1;Q+n)=(1,a_1, \ldots , a_k; n,b_1, \ldots , b_k)=(1,A_{1,k};n,B_{1,k})$ (see Definition \ref{betternotation}). Consider the primary circular pairs whose first $i$ connections are the same as those of $(P;Q)$. By the construction of $\mc{P}$, there are at most two such primary circular pairs, which we denote by 
\begin{align*}
(P;Q)_{1}&=(A_{1,i},C_{i+1,m}; B_{1,i}, D_{i+1,m})\\
(P;Q)_{2}&=(A_{1,i}, E_{i+1,m'}; b_{1,i}, f_{i+1,m'}).
\end{align*}
By Lemma \ref{crossing} and the construction of $\mc{P}$, we have that $c_{i+1}\le a_{i+1}$ and $d_{i+1}\ge b_{i+1}$ or $e_{i+1}\le a_{i+1}$ and $f_{i+1} \ge b_{i+1}$ (or else we would have been able to set $d_{i+1}=b_{i+1}$ or $e_{i+1}=a_{i+1}$). Furthermore, exactly one of these pairs of inequalities holds. Let us assume that the former holds, as the latter case is identical, and in this case, call $(P;Q)_{1}$ the primary circular pair associated to $(P;Q)$. If, on the other hand, $(P;Q)_{1}$ (as above) is the only primary circular pair sharing its $i$ connections with $(P;Q)$, then $c_{i+1}\le a_{i+1}$ and $d_{i+1}\ge b_{i+1}$, and we still refer to $(P;Q)_{1}$ as the primary circular pair associated to $(P;Q)$.

We now prove the lemma by retrograde induction on $i$, where here $i$ is such that the first $i$ connections of $(P;Q)$ are shared with some primary circular pair. If $i=k$, we are done by Lemma \ref{Pa1Qbn}, and if the first $i$ connections of $(P;Q)$ are exactly the primary circular pair in question, we are done by the Subset Axiom. Otherwise, we first need $(A;B)=(A_{1,i}, c_{i+1},A_{i+2,k}; B_{1,i},d_{i+1},B_{i+2,k})\in S_{0}$, which follows from Lemmas \ref{trip} and \ref{rtol}, where the conditions of these lemmas are satisfied as a result of Lemma \ref{crossing}. It is easy to see that the primary circular pair associated to $(A;B)$ is the same as that for $(P;Q)$. It follows, then, by the inductive hypothesis, that $(1,A_{1,i},c_{i+1},A_{i+2,k}; B_{1,i},d_{i+1},B_{i+2,k},n)\in Z$. When $a_{i+1}\neq c_{i+1}$ and $b_{i+1}\neq d_{i+1}$, Lemma \ref{ltol} yields the desired result, and if one of $a_{i+1}=c_{i+1}$ or $b_{i+1}=d_{i+1}$, we are done by a similar argument.
\end{proof}

\begin{lemma}\label{unPCP}
There is exactly one circular pair in $\mc{P}$ that does not lie in $S_{0}$, which we call the \textbf{$S_{0}$-primary circular pair}.
\end{lemma}
\begin{proof}
By Lemma \ref{PCPcont}, $\mc{P}\setminus S_{0}$ has at least one element, because $S_{0}$ does not have the $(n,1)$-BEP. Assume, for sake of contradiction, that $\mc{P}\setminus S_{0}$ has two elements, of the form 
\begin{align*}
&(A_{1,i-1}, c,P; B_{1,i-1}, d, Q)\\
&(A_{1,i-1}, e,P'; B_{1,i-1}, f, Q').
\end{align*}
Because $(A_{1,i-1}, c,P; B_{1,i-1}, d, Q)\notin S_{0}$ and $(A_{1,i-1},P; B_{1,i-1}, Q),(A_{2,i-1}, c,P; B_{2,i-1}, d, Q)\in S_{0}$, we must have $(A_{2,i-1}, c,P; B_{1,i-1}, Q)\in S_{0}$, by Axiom \ref{axside}. Thus, $(A_{2,i-1}, c; B_{1,i-1})\in S_{0}$, by the Subset Axiom. By the same argument applied to $e,f$, we must have $(A_{1,i-1}; B_{2,i-1},f)\in S_{0}$, so Axiom \ref{axcross} gives $(A_{2,i-1}, c; B_{2,i-1},f)\in S_{0}$. However, because $f>d$, we have a contradiction of the definition of $d$. Thus, $|\mc{P}\setminus S_{0}|=1$.
\end{proof}

\begin{lemma}\label{contPCP}
For any incomplete circular pair $(P;Q)$, any electrical positroid $Z$ satisfying $S_{0}\cup\{(P+1;Q+n)\}\subset Z\subset S_{1}$ contains the $S_{0}$-primary circular pair.
\end{lemma}

\begin{proof}
Proceed by retrograde induction on $i$, where $i$ is such that the first $i$ connections of $(P;Q)$ are the same as those of some primary circular pair. By the Subset Axiom, we can assume that $(P;Q)$ is minimal. The base case is immediate from the Subset Axiom, so suppose that $i<k$. Let $(P;Q)=(A_{1,k};B_{1,k}).$ Then, we need to show that, if $(P+1;Q+n)\in Z$, then $(1,A_{1,i}, c_{i+1}, A_{i+2,k};n, B_{1,k})\in Z$.

First, suppose that both $c_{i+1}<a_{i+1}$ and $b_{i+1}< d_{i+1}$. Then, the desired claim is exactly Lemma \ref{rtol}, as long as $i+1<m$. Assume, then, that $i+1=m$. First, an application of Lemma \ref{rtol} yields $(A_{1,i}, c_{i+1},A_{i+2,k};B_{1,i}, d_{i+1},B_{i+2,k})\in S_{0}$, which implies $(A_{1,i}, c_{i+1}, A_{i+2,k},B_{1,k})\in S_{0}$ and $(A_{1,k};B_{1,i}, d_{i+1}, B_{i+2,k})\in S_{0}$. Furthermore, if $(A_{i,i}, c_{i+1}, A_{i+1,k}; n, B_{1,k})\in S_{0}$, then Axiom \ref{axleft} yields $(1,A_{1,k}; n, B_{1,k})\in S_{0}$, a contradiction. 

Similarly, we have $(1, A_{1,i}, c_{i+1}, A_{i+2,k}; B_{1,i}, d_{i+1}, B_{i+1,k})\notin S_{0}$. As a result, applying Axiom \ref{axmid} to $(1, A_{1,i}, c_{i+1}, A_{i+1,k}; n, B_{1,k})$ with $1,c_{i+1},a_{i+1},n$ gives $(1, A_{1,i}, c_{i+1}, A_{i+2,k}; n, B_{1,k})\in S_{0}$. One more application of Axiom \ref{axmid} to  $(1, A_{1,i}, c_{i+1}, A_{i+2,k}; n, B_{1,i}, d_{i+1}, B_{i+1},k)$ with $n,d_{i+1}, b_{i+1}, 1$ yields the desired result.

If one of the indices $a_{i+1}=c_{i+1}$ or $b_{i+1}=d_{i+1}$, then we are also done by a similar argument.
\end{proof}

\begin{corollary}\label{uniclos}For any two incomplete circular pairs $(P;Q)$ and $(P';Q')$, any electrical positroid $Z$ satisfying $S\cup\{(P+1;Q+n)\}\subset Z\subset S_{1}$ must also contain $(P'+1;Q'+n)$.
\end{corollary}
\begin{proof}
By Lemma \ref{contPCP}, $Z$ must contain the $S_{0}$-primary circular pair. The claim then follows by Lemma \ref{PCPcont}.
\end{proof}

By the above results, if we start with our set $S_{0}$ and some incomplete circular pair $(P;Q)\in S_{0}$, ``completing'' $(P;Q)$ by adding $(P+1;Q+n)$ to $S_{0}$ will require that we have completed every incomplete pair. We now finish the proof of Theorem \ref{circposthm}, in the boundary edge case.

Let $T_{0}\subset S_{0}$ denote the subset of circular pairs in $S_{0}$ without the connection $(1,n)$, and define $T_{1},T'_{0}$ similarly for $S_{1},S'_{1}$, respectively. By construction, it is easily seen that $T_{0}=T_{1}=T'_{0}$. While $T_{0}$ may not necessarily be an electrical positroid, we have:

\begin{lemma}\label{findclos}
There exists an electrical positroid $T$ with $T_{0}\subset T\subset S_{0}\cap S'_{0}$.
\end{lemma}

\begin{proof}
We give an algorithm to construct such an electrical positroid $T$. We begin by setting $T=T_{0}$; note that $T$ satisfies the last two electrical positroid axioms, but may not satisfy the first six. Each of the first six axioms are of the form $\mc{A},\mc{B}\in T\Rightarrow \mc{C},\mc{D}\in T$, or otherwise $\mc{A},\mc{B}\in T\Rightarrow \mc{C},\mc{D}\in T$ or $\mc{E},\mc{F}\in T$. At each step of the algorithm, if $T$ is an electrical positroid, we stop, and if not, we pick an electrical positroid axiom $\alpha$ (among the first six) not satisfied by $\mc{A},\mc{B}\in T$. We then show that we can add elements of $S_{0}\cap S'_{0}$ to $T$ so that $\alpha$ is satisfied by $\mc{A},\mc{B}$, and so that $T$ also still satisfies the Subset Axiom.

It is clear that adding circular pairs to $T$ in this way is possible when $\alpha$ is one of Axioms \ref{axcross}, \ref{axsep}, \ref{axleft}, and \ref{axright}: we take the add circular pairs $\mc{C},\mc{D}$, as above, as well as all of the circular pairs formed by subsets of their respective connections. We will show that this operation is also possible when $\alpha$ is one of Axioms \ref{axside} and \ref{axmid}. From here, it will be clear that the algorithm must terminate, because we can only add finitely many elements to $T$. Therefore, we will eventually find $T$ with the desired properties.

In each of the cases below, the circular pairs added to $T$ are always assumed to be added along with each of their subsets, that is, the circular pairs formed by subsets of their connections. In this way, the Subset Axiom is satisfied by $T$ at all steps in the algorithm.

We first consider Axiom \ref{axside}, which we assume to fail in $T$ when applied to $(P-a; Q-c),(P-b; Q-d)\in T$. If $(P-a; Q-c),(P-b; Q-d)\in S_{0}$, either $(P-a;Q-d),(P-b;Q-c)\in S_{0}$ or $(P-a-b;Q-c-d),(P;Q)\in S_{0}$. It is easy to see that $1\in P$ and $n \in Q$ (or vice versa, but we can swap $P$ and $Q$ and reverse their orders), or else Axiom \ref{axside} already would have been satisfied by $(P-a; Q-c),(P-b; Q-d)\in T$. We proceed by casework:

\begin{itemize}
\item $(a,c)=(1,n)$. Then, because Axiom \ref{axside} fails, we have $(P-b;Q-c)\notin T_{0},S_{0},S'_{0}$. Thus, we may add $(P-a-b;Q-c-d),(P;Q)$ to $T$, and these lie in $S\cap S''$.

\item $a=1, c\neq n$. We have $(P-a;Q-c),(P-b;Q-d)\in T$. First, suppose that $(P-a;Q-d)\notin T$. Because $(P-a;Q-d)$ does not contain the connection $(1,n)$, we have $(P-a;Q-d)\notin S_{0},S'_{0}$. Then, $(P-a-b;Q-c-d),(P;Q)\in S_{0}, S'_{0}$, so we may add may $(P;Q)$ to $T$, so that Axiom \ref{axside} is satisfied with $(P-a;Q-c),(P-b;Q-d)\in T$ (note that $(P-a-b;Q-c-d)$ is already in $T$).

Now, suppose instead that $(P-a;Q-d)\in T$. If $(P-b-1;Q-c-n)\notin T$, then $(P-b;Q-c)\notin S_{0}, S'_{0}$ by the Subset Axiom. Then, $(P-a-b;Q-c-d),(P;Q)\in S_{0}, S'_{0}$, and so we may add $(P;Q)$ to $T$ to satisfy Axiom \ref{axside}. Now, assume that $(P-b-1;Q-c-n)\in T$. For any electrical positroid $\overline{S}$, Axiom \ref{axleft} applied to $(P-b;Q)$ and $d,c,n,1$ gives that $(P-b;Q-d)\in \overline{S}$ and $(P-b-1;Q-c-n)\in \overline{S}$ implies $(P-b;Q-c)\in \overline{S}$ and $(P-b-1;Q-n-d)\in \overline{S}$. By the discussion above, we have $(P-b;Q-d),(P-b-1;Q-c-n)\in T,S_{0},S'_{0}$, and so we may add $(P-b;Q-c)$ to $T$. The case in which $a\neq 1, c=n$ is identical.

\item The case $a\neq 1, c\neq n$ may be handled using similar logic; the details are left to the reader.

\end{itemize}

Finally, consider Axiom \ref{axmid}, which we assume to fail for $(P-b;Q),(P-a-c;Q-d)\in T$. As before, we may assume $1\in P,n\in Q$.
\begin{itemize}
\item $(a,d)=(1,n)$, or $(d,a)=(n,1)$. Similar to the first case above.
\item
$a=1,d\neq n$. Then, we have $(P-b;Q),(P-a-c;Q-d)\in T,S_{0},S'_{0}$. As in the second case for Axiom \ref{axside}, we may assume that we have $(P-a;Q)\in T,S_{0},S'_{0}$ and $(P-a-b;Q-d)\in T,S_{0},S'_{0}$, or else both $S_{0}$ and $S'_{0}$ would contain exactly one of $(P-a;Q),(P-b-d;Q-d)$ and $(P-c;Q),(P-a-b;Q-d)$. Moreover, we may assume that we have $(P-1-b-c;Q-n-d)\in T,S_{0},S'_{0}$ by similar logic. Because $(P-b;Q)\in T,S_{0},S'_{0}$ and $(P-1-b-c;Q-n-d)\in T,S_{0},S'_{0}$, we may apply Axiom \ref{axsep} to find that $(P-b-c;Q-d)\in S_{0},S'_{0}$. Thus, we can add $(P-b-c;Q-d)$ to $T$, so that we still have $T\subset S_{0}\cap S'_{0}$. The case $a=n, d\neq 1$ is identical.

\item The cases $a\neq 1, d=n$ and $a\neq 1, d\neq n$ may be handled using similar logic; we again omit the details.

\end{itemize}
Thus, in all cases, our algorithm is well-defined, and we are done.
\end{proof}

\begin{proof}[Proof of Theorem \ref{circposthm}] By Lemma \ref{uniclos}, we must in fact have $S_{0}=T=S'_{0}$, provided that neither $S_{0}$ nor $S'_{0}$ is equal to $S_{1}$, which is true by construction (recall that $G'_{0}$ is critical). The proof is complete, in the boundary edge case.

It is left to consider the case in which $S_{0}$ has the $(i,i+1)$-BEP, for each $i$, but fails to have the $i$-BSP, for some $i$. Without loss of generality, suppose that $S_{0}$ does not have the $1$-BSP.  We now form $S_{1}$ as the union of $S_{0}$ and the set of all circular pairs $(P+x;Q+y)$ such that $(P+x;Q+1),(P+1,Q+y)\in S_{0}$, where $(P;Q)$ is a circular pair with $1,x\notin P,1,y\notin Q$.

In Appendix \ref{BSPcase}, we form a circular planar graph $G_{1}$ such that $\pi(G_{1})=S_{1}$ and $G_{1}$ has a boundary spike at 1. Then, contracting this boundary spike to obtain the graph $G_{0}$, we find that $\pi(G_{0})=S_{0}$. Therefore, with the additional results of Appendix \ref{BSPcase}, the theorem is proven.
\end{proof}

\section{The LP Algebra $\mc{LM}_{n}$}\label{lpsec}

We now study the LP Algebra $\mc{LM}_{n}$. Our starting point will be \textit{positivity tests}; a particular positivity test will form the initial seed in $\mc{LM}_{n}$. We then proceed to investigae the algebraic and combinatorial properties of clusters in $\mc{LM}_{n}$.

\subsection{Positivity Tests}

Let $M$ be a symmetric $n\times n$ matrix with row and column sums equal to zero. In this section, we describe tests for deciding if $M$ is the response matrix for an electrical network $\Gamma$ in the top cell of $EP_{n}$, defined in \cite[\S 3]{paper1}. That is, we describe tests for deciding if all of the circular minors of $M$ are positive. These tests are similar to certain tests for total positivity described in \cite{totalpos}. Throughout the remainder of this section, all indices around the circle are considered modulo $n$, and we will refer to circular pairs and their corresponding minors interchangeably.

\begin{defn}
A set $S$ of circular pairs is a \textbf{positivity test} if, for all matrices $M$ whose minors corresponding to $S$ are positive, every circular minor of $M$ is positive (equivalently, $M$ is the response matrix for a top-rank electrical network).
\end{defn}

We begin by describing a positivity test of size $\binom{n}{2}$. Fix $n$ vertices on a boundary circle, labeled $1,2,\ldots,n$ in clockwise order.

\begin{defn}
For two points $a,b\in[n]$, let $d(a,b)$ denote the number of boundary vertices on the arc formed by starting at $a$ and moving clockwise to $b$, inclusive.
\end{defn}

\begin{defn}
A circular pair $(P;Q) = (p_1, \cdots, p_k; q_1, \cdots, q_k)$ is called \textbf{solid} if both sequences $p_1,\ldots,p_k$ and $q_1,\ldots,q_k$ appear consecutively in clockwise order around the circle. Write $d_1 = d_1(P;Q)= d(p_k,q_k)$, and $d_2 = d_2(P;Q)= d(q_1,p_1)$. We will call a solid circular pair $(P;Q)$ \textbf{picked} if one of the following conditions holds:
\begin{itemize}
\item $d_1 \leq d_2$ and $1 \leq p_1 \leq \frac{n}{2}$, or
\item $d_1 \geq d_2$ and $1 \leq q_k \leq \frac{n}{2}$
\end{itemize}
\end{defn}

\begin{defn}
Let $M$ be a fixed symmetric $n\times n$ matrix. Define the set of \textbf{diametric pairs} $\mc{D}_{n}$ to be the set of solid circular pairs $(P;Q)$ such that either $|d_1 - d_2| \leq 1$ or $|d_1 - d_2| = 2$ and $(P;Q)$ is picked. We will refer to the elements of $\mc{D}_{n}$ as circular pairs and minors interchangeably.
\end{defn}

It is easily checked that $|\mc{D}_{n}|=\binom{n}{2}$.

\begin{rmk}
For a solid circular pair $(P;Q)$, we have that $|d_1 - d_2| \equiv n \pmod{2}$, so $\mc{D}_{n}$ consists of the solid circular pairs with $|d_1 - d_2| = 1$ when $n$ is odd, and the solid circular pairs with either $|d_1 - d_2| = 0$, or $|d_1 - d_2| = 2$ and $(P;Q)$ is picked when $n$ is even.
\end{rmk}

Recall (see Remark \ref{flipPQ}) that the circular pairs $(P;Q)$ and $(\widetilde{Q};\widetilde{P})$ will be regarded as the same. Note, for example, that $(P;Q)\in\mc{D}_{n}$ if and only if $(\widetilde{Q};\widetilde{P})\in\mc{D}_{n}$, so the definition of $\mc{D}_{n}$ is compatible with this convention.

\begin{proposition} \label{initialseed}
If $M$ is taken to be an $n\times n$ symmetric matrix of indeterminates, any circular minor is a positive rational expression in the determinants of the elements of $\mc{D}_{n}$.
\end{proposition}
\begin{proof}
We will make use of the Grassmann-Pl\"{u}cker relations, (\ref{P1}) and (\ref{P2}), so we repeat them here:

For $(a,b;c,d)$ a circular pair,
\begin{align}\tag{\ref{P1}} \Delta^{a,c}\Delta^{b,d}=\Delta^{a,d}\Delta^{b,c}+\Delta^{ab,cd}\Delta^{\emptyset,\emptyset} \end{align}
and for $a,b,c,d$ in clockwise order,
\begin{align}\tag{\ref{P2}} \Delta^{b,\emptyset}\Delta^{ac,d}=\Delta^{a,\emptyset}\Delta^{bc,d}+\Delta^{c,\emptyset}\Delta^{ab,d}\end{align}

We will first show, by induction on $|d_{1}-d_{2}|$, that any solid circular pair is a positive rational expression in the elements of $\mc{D}_{n}$. There is nothing to check when $|d_{1}-d_{2}|$ is equal to 0 (when $n$ is even) or 1 (when $n$ is odd). If $(P;Q) = (p_1,\ldots,p_k; q_1,\ldots,q_k)$ is a solid circular pair such that $|d_1 - d_2| = 2$ (hence, $n$ is even) and $(P;Q)$ is not picked, then either $(p_1 - 1, p_1, \ldots, p_k; q_1 + 1, q_1, \ldots, q_k)$ or $(p_1, \ldots, p_k, p_k + 1; q_1, \ldots, q_k, q_k - 1)$ is a solid circular pair with $|d_1 - d_2| = 2$, and it must be picked. Assume, without loss of generality, that it is the former, and let $p_0 = p_1 - 1$, $q_0 = q_1 + 1$. Letting $\Delta = (p_0, \ldots, p_k; q_0, \ldots, q_k)$, we have, by (\ref{P1}), that:

\begin{align} \label{apply1}
\Delta^{p_0, q_0} = \frac{\Delta^{p_0, q_k} \Delta^{p_k,q_0} + \Delta^{p_0 p_k, q_0 q_k}   \Delta^{\emptyset, \emptyset}}{\Delta^{p_k,q_k}}.
\end{align}

Because $(p_1,\ldots,p_k; q_1,\ldots,q_k) = \Delta^{p_0, q_0}$ is not picked, $\Delta^{p_k,q_k}$ corresponds to a picked circular pair with $|d_1 - d_2| = 2$, and $\Delta^{p_0, q_k}, \Delta^{p_k,q_0}, \Delta^{p_0 p_k, q_0 q_k},$ and $\Delta^{\emptyset, \emptyset}$ all have $|d_1 - d_2| = 0$, so we have that $(P;Q)$ is a rational expression of elements of $\mc{D}_{n}$.

Now, for $m \geq 3$, assume that all solid pairs with $|d_1 - d_2| < m$ are positive rational expression in the elements of $\mc{D}_{n}$, and consider a solid pair $(P;Q)$ with $|d_1 - d_2| = m$. Then, either $(p_1 - 1, p_1, \ldots, p_k; q_1 + 1, q_1, \ldots, q_k)$ or $(p_1, \ldots, p_k, p_k + 1; q_1, \ldots, q_k, q_k - 1)$ is a solid circular pair with $|d_1 - d_2| < m$. Assume, without loss of generality, that the former is the case. Then, we may again set $p_0 = p_1 - 1$, $q_0 = q_1 + 1$, and $\Delta=(p_0, \ldots, p_k; q_0, \ldots, q_k)$, and (\ref{apply1}) still holds. Each term on the right hand side corresponds to a solid pair with a smaller value of $|d_1 - d_2|$, so $(P;Q)$ is a positive rational expression in the elements of $\mc{D}_{n}$, by the inductive hypothesis.

We now show that any circular pair is a positive rational expression in the elements of $\mc{D}_{n}$. For a sequence $P = p_1, \ldots, p_k$ of points ordered clockwise around the circle, let $c_P \in \{1,\ldots,k\}$ be the largest index such that $p_1, \ldots, p_{c_P}$ are consecutive. If $c_P < k$, then define $d_3(P) = d(p_{c_P}, p_{c_P + 1})$ and $d_4(P) = d(p_1, p_k)$. If, on the other hand, $c_P = k$, define $d_3(P) = d_4(P) = 0$. Similarly, for a sequence $Q = q_1, \ldots, q_k$ of points ordered counterclockwise around the circle, let $c_Q \in \{1,\ldots,k\}$ be the smallest index such that $q_k, \ldots, q_{c_Q}$ are consecutive. If $c_P  >1$, then define $d_3(Q) = d(p_{c_Q + 1}, p_{c_Q})$ and $d_4(Q) = d(q_k, q_1)$. If, on the other hand, $c_Q = 1$, define $d_3(P) = d_4(P) = 0$.

For a circular pair $(P;Q)$, define $\Phi((P;Q)) = d_3(P) + d_4(P) + d_3(Q) + d_4(Q)$, and note that $\Phi((P;Q))=\Phi((\widetilde{Q};\widetilde{P}))$, so $\Phi$ is well-defined when we impose the convention $(P;Q)=(\widetilde{Q};\widetilde{P})$. We finish the proof by showing that any circular pair is a positive rational expression in the elements of $\mc{D}_{n}$ by induction on $\Phi$.

If $\Phi((P;Q)) = 0$, then $(P;Q)$ is solid and hence a rational expression in the elements of $\mc{D}_{n}$. Now, assume that for any $m>0$, every circular pair $(P';Q')$ with $\Phi((P';Q')) < m$ is a rational expression in the elements of $\mc{D}_{n}$, and consider a circular pair $(P;Q)$ with $\Phi((P;Q)) = m$. Assume, without loss of generality, that $c_{P}\neq k$, and let $\ell = p_{c_{P}} +1$. Applying (\ref{P1}) to $\Delta = (p_1, \ldots, p_{c_{P}}, \ell, p_{c_{P}+1}, \ldots, p_k; q_1, \ldots, q_k)$, we get:

\begin{align} \label{apply2}
\Delta^{\ell, \emptyset} = \frac{\Delta^{p_1, \emptyset} \Delta^{\ell p_k, q_k} + \Delta^{p_k, \emptyset}  \Delta^{p_1 \ell, q_k}}{\Delta^{p_1 p_k, q_k}}.
\end{align}
Each term on the right hand side is easily seen to have a smaller value of $\Phi$ than $m$, so we are done by induction.
\end{proof}

\begin{corollary}\label{postestcor}
$\mc{D}_{n}$ is a positivity test.
\end{corollary}

\subsection{$\mc{CM}_{n}$ and $\mc{LM}_{n}$}

The positive rational expressions from the previous section are reminiscent of a cluster algebra structure (see \cite[\S 3]{fominpasha} for definitions). In fact, (\ref{P1}) and (\ref{P2}) are exactly the exchange relations for the local moves in double wiring diagrams \cite[Figure 9]{totalpos}. Due to parity issues similar to those encountered when attempting to associate a cluster algebra to a non-orientable surface in \cite{dp}, the structure of positivity tests is slightly different from a cluster algebra. We present the structure in two different ways: first, as a Laurent phenomenon (LP) algebra $\mc{LM}_{n}$ (see \cite[\S 2,3]{lp} for definitions), and secondly as a cluster algebra $\mc{CM}_{n}$ similar to the double cover cluster algebra in \cite{dp}. $\mc{LM}_{n}$, we will find, is isomorphic to the polynomial ring on $\binom{n}{2}$ variables, but more importantly encodes the information of the positivity of the circular minors of a fixed $n\times n$ matrix.

We begin by describing an undirected graph $U_n$ that encodes the desired mutation relations among our initial seed. The vertex set of $U_n$ will be $V_n = \mc{D}_{n} \cup \{(\emptyset; \emptyset) \}$.

\begin{defn}
A solid circular pair $(p_1, \ldots,p_k; q_1, \ldots, q_k)$ is called \textbf{maximal} if $2k+2 > n$ or $2k+2=n$ and $d_1=d_2$. A solid circular pair $(P;Q)=(p_1, \ldots,p_k; q_1, \ldots, q_k)$ is called \textbf{limiting} if $|d_1 - d_2| = 2$, $(P;Q)$ is picked, and $1 = p_1$ or $1 = q_k$.
\end{defn}

Let us now describe the edges of $U_{n}$: see Figure \ref{Q8} for an example. For each $(P;Q) \in V_n$ that is not maximal, limiting, or empty (that is, equal to $\epair$), there is a unique way to substitute values in \ref{P1} such that $(P;Q)$ appears on the left hand side, and all four terms on the right hand side are in $V_n$. We draw edges from $(P;Q)$ to these four vertices in $U_n$. Finally, if $(P;Q), (R;S) \in V_n$ are limiting, we draw an edge between them if their sizes differ by $1$. The edges drawn in these two cases constitute all edges of $V_{n}$.

For any maximal circular pair $(P;Q)$, it can be proven that there exists a symmetric matrix $A$ such that $A$ is positive on any circular pair except $A(P;Q) \leq 0$. In fact, if $(P;Q)$ is maximal and has $|d_1 - d_2| \leq 1$, then the set of all circular pairs other than $(P;Q)$ is an electrical positroid. Hence, in our quivers, we will take the vertices corresponding to the maximal circular pairs and $\epair$ to be frozen.

$U_n$ can then be embedded in the plane in a natural way with the circular pairs of size $k$ lying on the circle of radius $k$ centered at $\epair$, and all edges except those between vertices corresponding to limiting circular pairs either along those circles or radially outward from $\epair$.

\begin{figure}
\begin{center}
\includegraphics[scale=1]{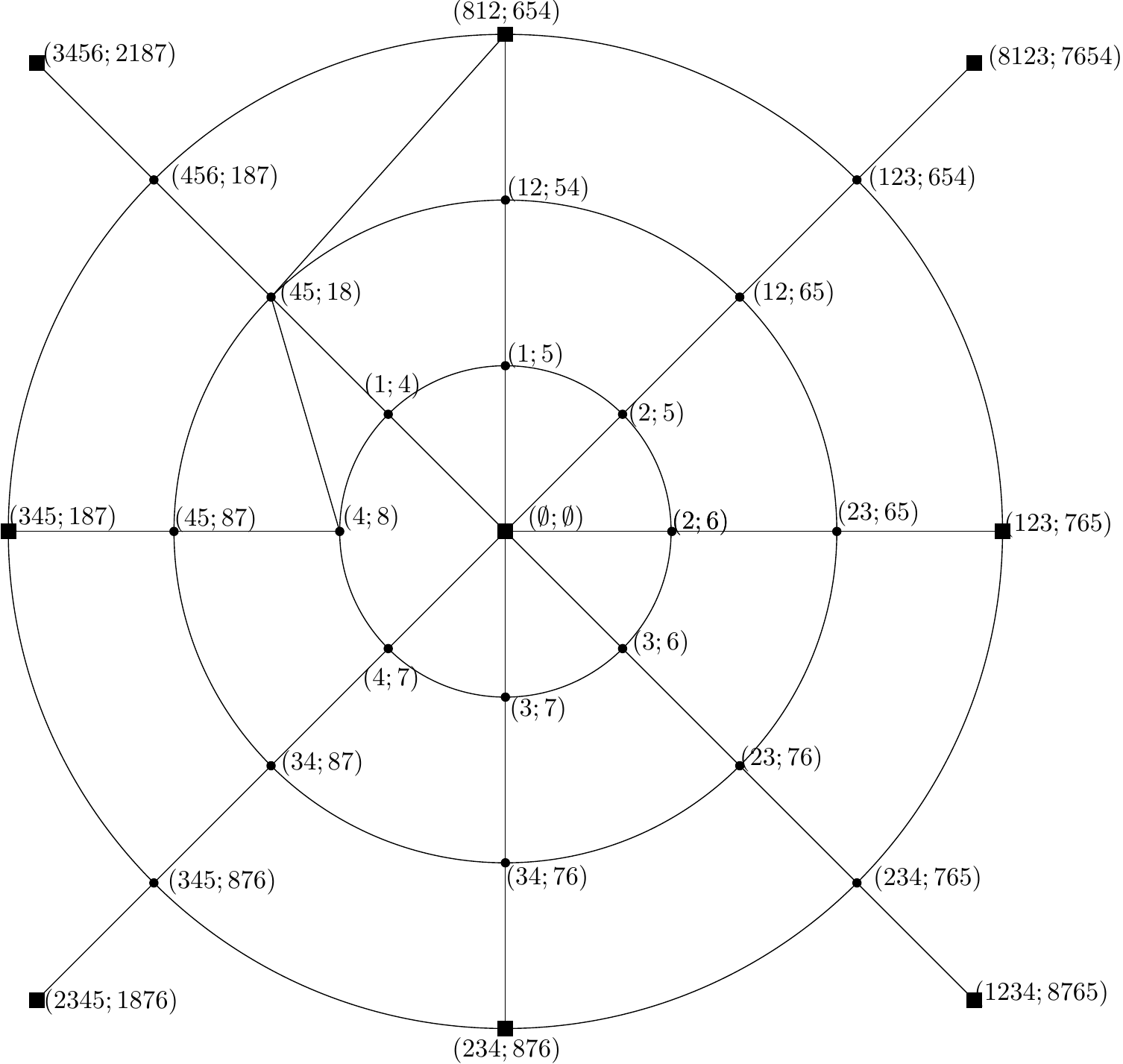}
\caption{The graph $U_8$ depicting the desired exhange relations among $\mc{D}_{8}$. Vertices marked as squares correspond to frozen variables. $(4;8)$, $(45;18)$ and $(812;654)$ are the limiting circular pairs.}\label{Q8}
\end{center}
\end{figure}

If we could orient the edges of $U_n$ such that they alternate between in- and out-edges at each non-frozen vertex, then the resulting quiver would give a cluster algebra such that mutations at vertices whose associated cluster variables are neither frozen not limiting correspond to the relation Grassmann-Pl\"{u}cker relation (\ref{P1}). Furthermore, these mutation relations among the vertices of $V_n$ constitute all of the Grassmann-Pl\"{u}cker relations for which five of the six terms on the right hand side are elements in $V_n$, and the term which is not in $V_{n}$ is on the left hand side of the relation (\ref{P1}) or (\ref{P2}). However, for $n \geq 5$, such an orientation of the edges of $U_{n}$ is impossible, because the dual graph of $U_{n}$ contains odd cycles. We thus define:

\begin{defn} \label{deflmn}
Let $\mc{LM}_{n}$ be the LP algebra constructed as follows: the initial seed $\mc{S}_n$ has cluster variables equal to the minors in $V_n$, with the maximal pairs and $\epair$ frozen, and, for any other $(P;Q) \in V_n$, the exchange polynomial $F_{(P;Q)}$ is the same as what is obtained from a quiver with underlying graph $U_{n}$, such that the edges around the vertex associated to $(P;Q)$ in $U_n$ alternate between in- and out-edges.
\end{defn}

For example, in $\mc{LM}_{8}$, the exchange polynomial associated to the cluster variable $x_{(12;54)}$ is $x_{(45;18)}x_{(12;65)}+x_{(1;5)}x_{(812;654)}$. We need the additional technical condition that $F_{(P;Q)}$ is irreducible as a polynomial in the cluster variables $V_{n}$, but the irreducibility is clear.

We next define a cluster algebra $\mc{CM}_{n}$ which is a double cover of positivity tests, in the following sense: we begin consider an $n\times n$ matrix $M'$, which we no longer assume to be symmetric. We write \textbf{non-symmetric circular pairs} in the row and column sets of $M'$ as $(P;Q)'$, so that $(P;Q)'$ and $(\widetilde{Q};\widetilde{P})'$ now represent different circular pairs. We will say that two expressions $A,B$ in the entries of $M'$ \textbf{correspond} if swapping the rows and columns for each entry in $A$ gives $B$, and we will write $B = c(A)$. For instance, $(P;Q)' = c((Q;P)')$.

The set of cluster variables $V'_{n}$ in our initial seed will consist of pairs $(P;Q)'$ such that $(P;Q) \in V_{n}$. Note that $|V'_{n}| = 2 \binom{n}{2} + 1$, as $V'_{n}$ contains $(P;Q)'$ and $(\widetilde{Q};\widetilde{P})'$ for each $(P;Q) \in \mc{D}_n$, and finally $\epair$. $(P;Q)'$ will be frozen in $V'_n$ if $(P;Q)$ was frozen in $V_n$.

We construct the undirected graph $U'_n$ with vertex set $V'_n$ by adding edges in the same way that $U_n$ was constructed. The only difference in our description is that  if $(P;Q), (R;S) \in V_n$ are limiting, then they will be adjacent only if their sizes differ by $1$ and $P \cap R \neq \emptyset$. See Figure \ref{Qp} for an example.

\begin{figure}
\begin{center}
\includegraphics[scale=0.7]{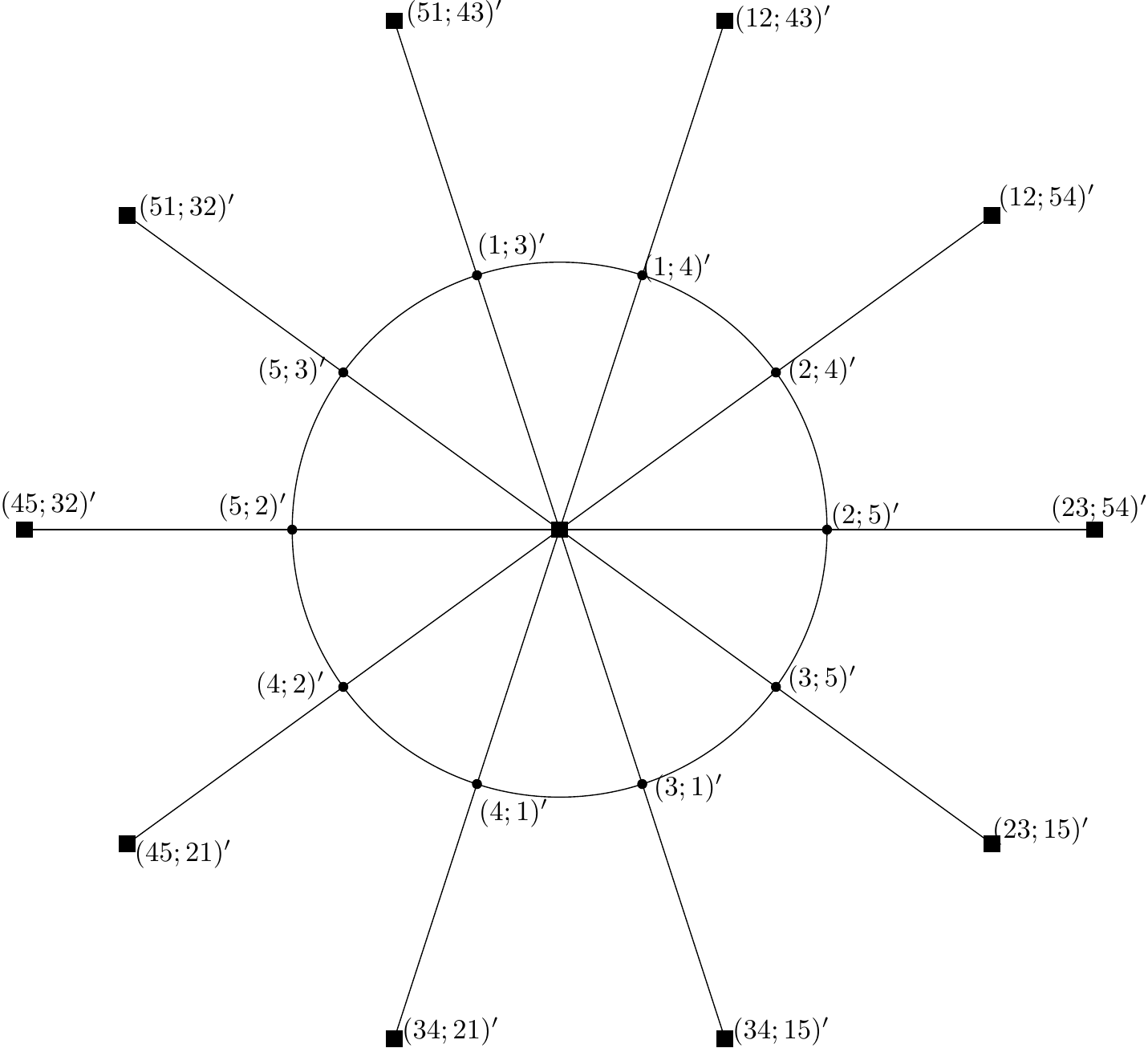}
\caption{The graph $U'_5$. In the quiver $\mc{Q}_{n}$, the edges alternate directions around a non-frozen vertex.}
\label{Qp}
\end{center}
\end{figure} 

Unlike in $U_n$, the edges of $U'_n$ can be oriented such that they are alternating around each non-frozen vertex. Let $\mc{Q}_n$ be the quiver from either orientation. Then, let $\mc{CM}_n$ be the cluster algebra with initial quiver $\mc{Q}_n$.

Breaking the symmetry of $M'$ removed the parity problems from $U_n$, so that we could define a cluster algebra, but we are still interested in using $U'_n$ to study $M$ when $M$ is symmetric. Toward this goal, we can restrict ourselves so that whenever we mutate at a cluster variable $v$, we then mutate at $c(v)$ immediately afterward. Call this restriction the \textbf{symmetry restriction}.

\begin{lemma} \label{symm}
After the mutation sequence $\mu_{x_1}, \mu_{c(x_1)}, \mu_{x_2}, \mu_{c(x_2)}, \ldots, \mu_{x_r}, \mu_{c(x_r)}$ from the initial seed in $\mc{CM}_n$, the number of edges from $x$ to $y$ in the quiver is equal to the number of edges from $c(y)$ to $c(x)$ for each $x,y$ in the final quiver.
\end{lemma}

\begin{proof}
We proceed by induction on $r$; for $r=0$, we have the claim by construction. Now, suppose that we have performed the mutations $\mu_{x_1}, \mu_{c(x_1)}, \mu_{x_2}, \mu_{c(x_2)}, \ldots, \mu_{x_{r-1}}, \mu_{c(x_{r-1})}$ and currently have the desired symmetry property. By the inductive hypothesis, $x_r$ and $c(x_r)$ are not adjacent, or else we would have had edges between them in both directions, which would have been removed after mutations. Thus, no edges incident to $c(x_r)$ are created or removed upon mutating at $x_r$. Hence, mutating at $c(x_r)$ afterward makes the symmetric changes to the graph, as desired.
\end{proof}

\begin{defn}
Let $\C[M]$ and $\C[M']$ denote the polynomial rings in the off-diagonal entries of $M$ and $M'$ respectively; recall that $M$ is symmetric, so $M_{ij} = M_{ji}$. Then, we can define the \textbf{symmetrizing homomorphism} $C : \C[M'] \to \C[M]$ by its action on the off-diagonal entries of $M'$:
\[C(M'_{ij}) = C(M'_{ji}) = M_{ij}.\]
If $S$ is a set of polynomials in $\C[M']$, then write $C(S) = \{ C(s) \mid s \in S \}$.
\end{defn}

\begin{lemma} \label{isom}
Let $L'_1$ be the cluster of $\mc{CM}_n$ that results from starting at the initial cluster and performing the sequence of mutations $\mu_{x_1}, \mu_{c(x_1)}, \mu_{x_2}, \mu_{c(x_2)}, \ldots, \mu_{x_r}, \mu_{c(x_r)}$. Let $L_2$ be the cluster of $\mc{LM}_n$ that results from starting at the initial cluster and performing the sequence of mutations $\mu_{x_1}, \mu_{x_2}, \ldots, \mu_{x_r}$. Then, $C(L'_1) = L_2$.
\end{lemma}

\begin{proof}
Using Lemma \ref{symm}, the proof is similar to \cite[Proposition 4.4]{lp}.
\end{proof}

In light of Lemma \ref{isom}, we may understand the clusters in $\mc{LM}_{n}$ by forming ``double-cover'' clusters in $\mc{CM}_{n}$. A sequence $\mu$ of mutations in $\mc{LM}_{n}$ corresponds to a sequence $\mu'$ of twice as many mutations in $\mc{CM}_{n}$, where we impose the symmetry restriction, and the cluster variables in $\mc{LM}_{n}$ after applying $\mu$ are the symmetrizations of those in $\mc{CM}_{n}$ after applying $\mu'$.

\begin{lemma} \label{LPpos}
Any cluster $S$ of $\mc{LM}_{n}$ consisting entirely of circular pairs is a positivity test.
\end{lemma}
\begin{proof}
In $\mc{CM}_n$, the exchange polynomial has only positive coefficients, so each variable in any cluster is a rational function with positive coefficients in the variables of any other cluster. In particular, each non-symmetric circular pair in $V_n-\epair'$ is a rational function with positive coefficients in the variables of any cluster reachable under the symmetry restriction. Hence, by Lemma \ref{isom}, each circular pair in $\mc{D}_n$ can be written as a rational function with positive coefficients of the variables in $S$. The desired result follows easily.
\end{proof}

As with double wiring diagrams for totally positive matrices \cite{totalpos}, and plabic graphs for the totally nonnegative Grassmannian \cite{sasha}, we now restricting ourselves to certain types of mutations in $\mc{LM}_n$. A natural choice is mutations with exchange relations of the form \ref{P1} or \ref{P2}. These mutations keep us within clusters consisting entirely of circular minors, the ``Pl\"{u}cker clusters.''

We begin by restricting ourselves only to mutations with exchange relations of the form \ref{P1}. Because the initial seed $\mc{S}_{n}$ consists only of solid circular pairs, we will only be able to mutate to other clusters consisting entirely of solid circular pairs. Our goal is to chracterize these clusters. We will be able to write down such a characterization using Corollary \ref{symallconsec} and Lemma \ref{isom}, and give a more elegant description of the clusters in Proposition \ref{wsconsec}.

\begin{defn}
Let $(P;Q)' = (p_1, \ldots, p_k; q_1, \ldots, q_k)'$ be a non-symmetric, non-empty circular pair. Define the statistics $D(P;Q)'$, $T(P;Q)'$, and $k(P;Q)'$ by:
\begin{align*}
D(P;Q)' &= d_1(P;Q)' - d_2(P;Q)' = d(p_k,q_k) - d(q_1,p_1)\\
T(P;Q)' &= \begin{cases}
\frac{p_1 + q_1}{2} \pmod{n} &\text{if } p_1 < q_1 \\
\frac{p_1 + q_1 + n}{2} \pmod{n} &\text{if } p_1 > q_1
\end{cases}\\
k(P;Q)' &= |P|\text{, that is, the size of }(P;Q)'
\end{align*}
\end{defn}

\begin{rmk}
A non-symmetric solid circular pair $(P;Q)'$ is uniquely determined by the triple $(D(P;Q)', T(P;Q)', k(P;Q)')$. A necessary condition for a triple $(D,T,k)$ to correspond to a non-symmetric solid circular pair is that $|D| + 2k \leq n$. When the terms are non-symmetric solid circular pairs, (\ref{P1}) can be written using these triples as:
\begin{equation}\label{P1t}
(D-2, T, k)(D+2, T, k)=(D,T-1/2 ,k)(D,T+1/2 ,k) + (D,T,k+1)(D,T,k-1).
\end{equation}
\end{rmk}

\begin{figure}[p]
\begin{center}
\includegraphics[scale=0.7]{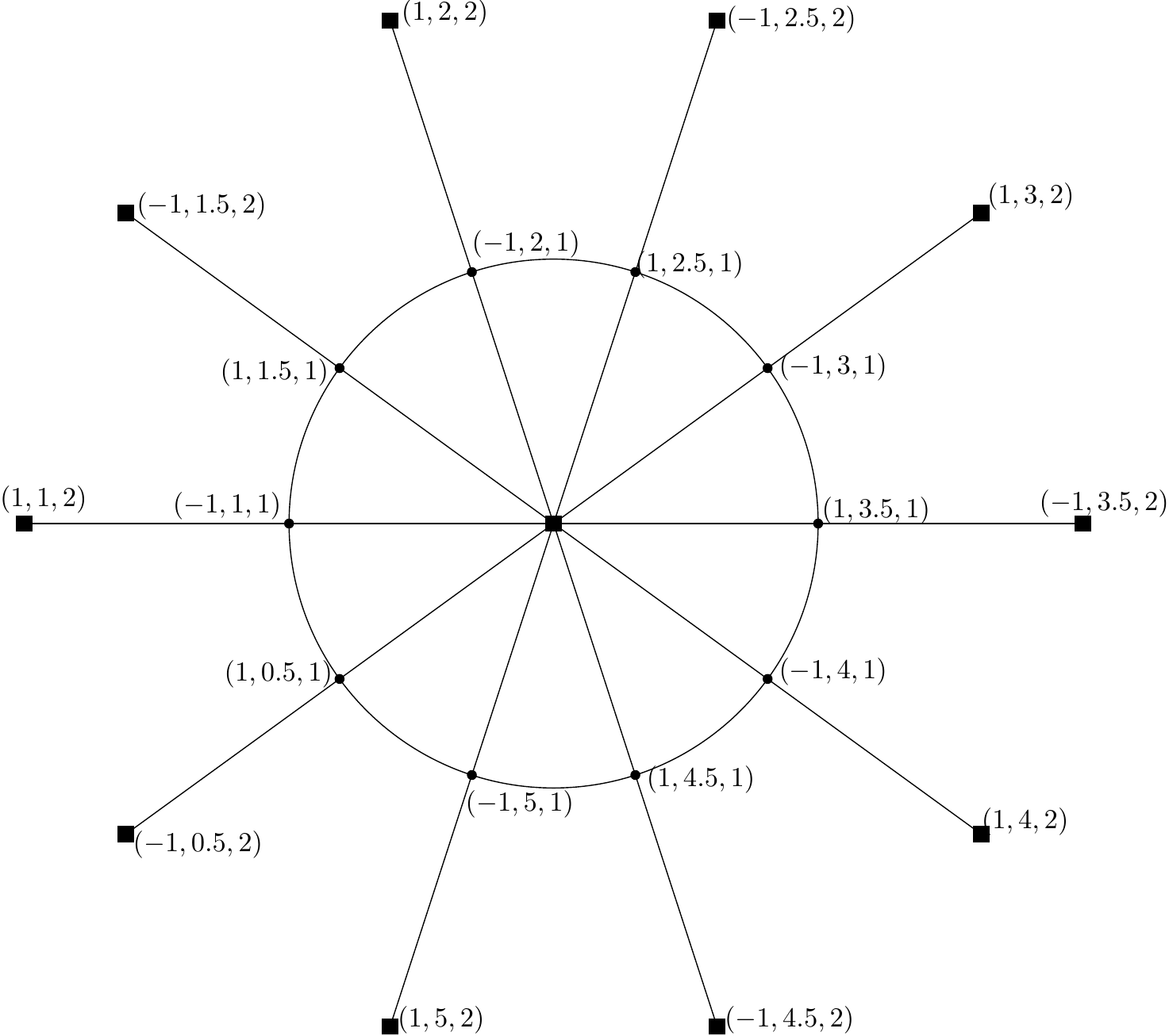}
\caption{The graph $U'_5$ with non-symmetric solid circular pairs labeled by triples $(D,T,k)$. In $\mc{Q}_{n}$, the edges alternate directions around each non-frozen vertex. Compare to Figure \ref{Qp}.}
\label{blues}
\vspace{0.4in}
\includegraphics[scale=0.7]{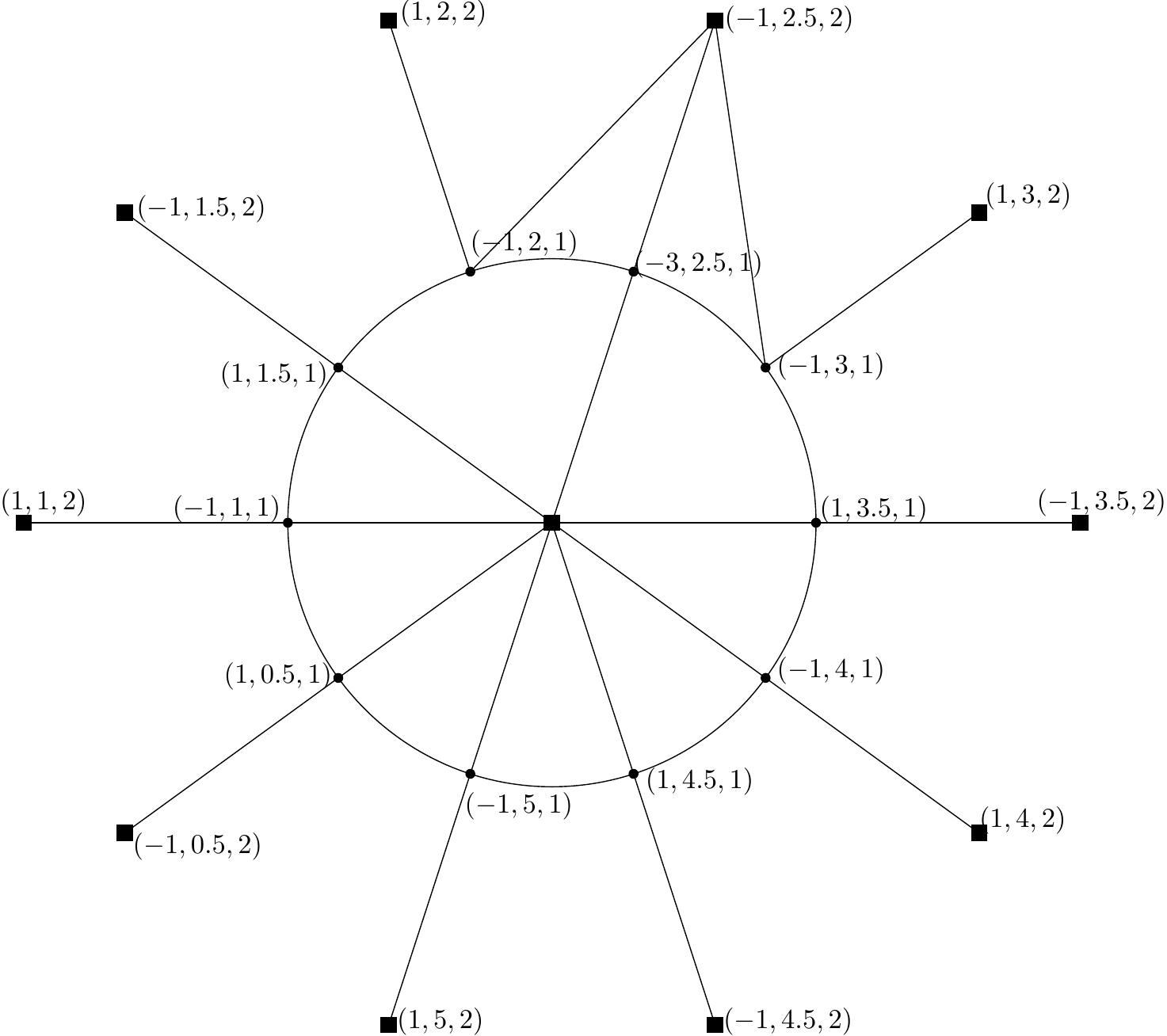}
\caption{The graph $U'_5$ after a mutation at $(1, 2.5, 1)$, with non-symmetric solid circular pairs labeled by triples $(D,T,k)$. In the quiver, the edges alternate directions around each non-frozen vertex.}
\label{bluesmutate}
\end{center}
\end{figure}

\begin{definition}
We call two non-symmetric solid circular pairs corresponding to the triples $(D_1,T_1,k_1)$ and $(D_2,T_2,k_2)$ \textbf{adjacent} if $T_1 = T_2$ and $|k_1 - k_2|= 1$, or $k_1 = k_2$ and $T_1 - T_2 \equiv \pm 1/2\pmod{n}$. We call $(P;Q)'$ and $(R;S)'$ \textbf{diagonally adjacent} if there are two non-symmetric solid circular pairs $(A;B)', (C;D)'$ which are both adjacent to both $(P;Q)'$ and $(R;S)'$. We call $(A;B)', (C;D)'$ the \textbf{connection} of $(P;Q)', (R;S)'$.
\end{definition}

Note that, in the initial quiver $\mc{Q}_{n}$, adjacent and diagonally adjacent circular pairs correspond to vertices which are adjacent in particular ways. Specifically, adjacent circular pairs correspond to vertices which are adjacent on the same concentric circle, or along the same radial spoke of $U'_{n}$. Diagonally adjacent circular pairs correspond to those which are adjacent via all other edges, the ``diagonal'' edges. We can now classify clusters of $\mc{CM}_{n}$ which can be reached only using the mutations with exchange relation (\ref{P1}).

\begin{definition} \label{solidcluster}
We call a set $S$ of $2 \binom{n}{2}+1$ non-symmetric solid circular pairs a \textbf{solid cluster} if it has the following properties:
\begin{itemize}
\item $\epair' \in S$,
\item for each integer $1 \leq k \leq \frac{n}{2}$, and each $T \in \{ 0.5, 1, 1.5, 2, \ldots, n \}$, unless $k=\frac{n}{2}$ and $T$ is an integer, there is a $D$ such that the non-symmetric solid circular pair corresponding to $(D,T,k)$ is in $S$, and
\item if $(P;Q)', (R;S)' \in S$ and $(P;Q)'$ is adjacent to $(R;S)'$, then $|D(P;Q)' - D(R;S)'| = 2$.
\end{itemize}
\end{definition}

\begin{remark} \label{embed}
There is a natural embedding of a solid cluster $S$ in the plane, similar to our embedding of $U'_n$. We place $\epair'$ at any point, and then pairs of size $k$ on the circle of radius $k$ centered at that point. Moreover, we place adjacent pairs of the same size consecutively around each circle, and adjacent pairs of different sizes collinear with $\epair'$.
\end{remark}

\begin{definition}
For a solid cluster $S$ of $\mc{CM}_{n}$ and associated quiver $\mc{B}$, we call $(S,B)$ a \textbf{solid seed} if it has the following properties:
\begin{itemize}
\item vertices corresponding to maximal non-symmetric solid circular pairs are frozen,
\item there is an edge between any pair of adjacent vertices that are not both frozen,
\item there is an edge between diagonally adjacent vertices $(P;Q)', (R;S)'$ if their connection $(A;B)', (F;G)'$ satisfies $|D(A;B)' - D(F;G)'| = 4$, 
\item there is an edge from a size 1 vertex $(P;Q)'$ to $\epair'$ if it would make the degree of $(P;Q)'$ even,
\item all edges of $\mc{B}$ are in drawn in one of the four ways described above, and
\item all edges are oriented so that, in the embedding described in Remark \ref{embed}, edges alternate between in- and out-edges around any non-frozen vertex.
\end{itemize}

If, furthermore, $s \in S$ if and only if $C(s) \in S$, or equivalently, the non-symmetric solid circular pair corresponding to $(D,T,k)$ is in $S$ if and only if that corresponding to $(-D,T,k)$ is, then we call $(S,b)$ a \textbf{symmetric solid seed}.
\end{definition}

See, for example, Figure \ref{bluesmutate}.

\begin{remark} \label{canmut}
In a solid seed $(S,\mc{B})$, a variable $(P;Q)' \in S$ has an exchange polynomial of the form (\ref{P1}) whenever its corresponding vertex in $\mc{B}$, and has edges to the vertices corresponding to its four adjacent variables in $\mc{B}$, and no other vertices.
\end{remark}

\begin{lemma} \label{allconsec}
In $\mc{CM}_n$, from the initial seed with cluster $V'_{n}$ and quiver $\mc{Q}'_n$, mutations of the form (\ref{P1}) may be applied to obtain the seed $(W'_n, \mc{R}'_n)$ if and only if $(W'_n, \mc{R}'_n)$ is a solid seed. Here, we do \emph{not} impose the symmetry restriction.
\end{lemma}

\begin{proof}
First, assume that $(W'_n, \mc{R}'_n)$ via mutations of the form (\ref{P1}). First, it is easy to check that $(V'_n, \mc{Q}'_n)$ is a solid seed. Then, our mutations do indeed turn non-symmetric solid circular pairs into other non-symmetric solid circular pairs. Furthermore, when we perform a mutation of the form (\ref{P1t}) at the vertex $v$, the values of $T$ and $k$ do not change, and the value of $D$ changes from being either 2 more than the values of $D$ at the vertices adjacent to $v$ to being 2 less, or vice versa. Hence, the resulting seed is also solid, so, by induction, $(W'_n, \mc{R}'_n)$ is solid.

Conversely, assume $(W'_n, \mc{R}'_n)$ is solid. We begin by noting that, by Remark \ref{canmut}, whenever the four terms on the right hand side of (\ref{P1t}) and one term on the left hand side are in our cluster, then we can perform the corresponding mutation.

Now, define $(I'_n, \mc{QI}'_n)$ to be the unique symmetric solid seed such that, for each $(P;Q)' \in I'_n$, $D(P;Q)' \in \{-2, -1, 0, 1, 2\}$. For any solid seed $(W'_n, \mc{R}'_n)$, we give a mutation sequence $\mu_{(W'_n, \mc{R}'_n)}$ using only mutations of the form (\ref{P1t}) that transforms $(W'_n, \mc{R}'_n)$ into $(I'_n, QI'_n)$. Hence, we will be able to get from the seed $(V'_n, \mc{Q}'_n)$ to $(W'_n, \mc{R}'_n)$ by performing $\mu_{(V'_n, \mc{Q}'_n)}$, followed by $\mu_{(W'_n, \mc{R}'_n)}$ in reverse order.

It is left to construct the desired mutation sequence. We define $\mu_{(W'_n, \mc{R}'_n)}$ as follows: while the current seed is not $(I'_n, \mc{QI}'_n)$, choose a vertex $v$ of the quiver, with associated cluster variable $(P;Q)'$, for which the value of $|D(P;Q)'|$ is maximized. We must have $|D(P;Q)'| > 2$, and by maximality, for each vertex $(R;S)'$ adjacent to $(P;Q)'$, we must have $|D(R;S)'| = |D(P;Q)'| - 2$. Hence, we can mutate at $(P;Q)$ to reduce $|D(P;Q)'|$ by at least 2. This process may be iterated to decrease the sum, over all cluster variables $(P;Q)'$ in our seed, of the $|D(P;Q)'|$, until we reach the seed $(I'_n, \mc{QI}'_n)$. The proof is complete.
\end{proof}

\begin{corollary} \label{symallconsec}
In $\mc{CM}_n$, from our initial seed with cluster $V'_{n}$ and quiver $\mc{Q}'_n$, we can apply symmetric pairs of mutations (that is, mutations with the symmetry restriction) of the form (\ref{P1}) to obtain the seed $(W'_n, \mc{R}'_n)$ if and only if $(W'_n, \mc{R}'_n)$ is a symmetric solid seed.
\end{corollary}

\begin{proof}
The proof is almost identical to that of Lemma \ref{allconsec}. It suffices to note that in a symmetric solid seed, $(P;Q)'$ has a maximal value of $|D|$ if and only if $c(P;Q)'$ does, so the mutation sequence $\mu_{(W'_n, \mc{R}'_n)}$ can be selected to obey the symmetry restriction.
\end{proof}

Using Corollary \ref{symallconsec} and Lemma \ref{isom}, we could also prove a similar result for the LP algebra $\mc{LM}_n$. However, we will wait to do so until Lemma \ref{wsconsec} where we will describe it more elegantly using the notion of weak separation.

We have now have the required machinery to prove our main theorem of this section.

\begin{theorem}\label{lpalg}
Fix a symmetric $n\times n$ matrix $M$ of distinct indeterminates. Then, the $\mathbb{C}$-Laurent phenomenon algebra $\mc{LM}_{n}$ is isomorphic to the polynomial ring (over $\mathbb{C}$) on the $\binom{n}{2}$ non-diagonal entries of $M$.
\end{theorem}

\begin{proof}
We may directly apply the analogue of \cite[Proposition 3.6]{fominpasha} for LP algebras (for which the proof is identical) to $\mc{LM}_n$ as defined in Definition \ref{deflmn}. It is well-known that minors of a matrix of indeterminates are irreducible, so we immediately have that all of our seed variables are pairwise coprime. We also need to check that each initial seed variable is is coprime to the variable obtained by mutating its associated vertex in $\mc{Q}_{n}$, and that this new variable is in the polynomial ring generated by the non-diagonal entries of $M$. For non-limiting (and non-frozen) minors, this is clear, because each such mutation replaces a minor with another minor via (\ref{P1}). For limiting minors, this is not the case, and we defer the proof to Appendix \ref{nonplucker}.

It remains to check, then that each of the $n(n-1)$ non-diagonal entries of $M'$ appear as cluster variables in some cluster of $\mc{CM}_{n}$. However, because $1\times1$ minors are solid, the result is immediate from Lemma \ref{symallconsec} and Lemma \ref{isom}.
\end{proof}

Because all $1\times1$ minors are solid, mutations of the form (\ref{P1}) were sufficient to establish Theorem \ref{lpalg}. Once we allow mutations of the form (\ref{P2}), the clusters become more difficult to describe. However, let us propose the following conjecture, which has been established computationally for $n\le6$.

\begin{conjecture} \label{getall}
Every cluster of $\mc{LM}_n$ consisting entirely of circular pairs can be reached from the initial cluster using only mutations of the form (\ref{P1}) or (\ref{P2}).
\end{conjecture}

\subsection{Weak Separation}

We next introduce an analogue of weakly separated sets from \cite{leclerc} for circular pairs. We recall the definition used in \cite{scott}, \cite{osp}, which is more natural in this case\footnote{Our definition varies slightly from that in the literature in the case where the two sets do not have the same size.}:

\begin{defn}
Two sets $A,B \subset [n]$ are \textbf{weakly separated} if there are no $a,a' \in A \setminus B$ and $b,b' \in B \setminus A$ such that $a<b<a'<b'$ or $b<a<b'<a'$.
\end{defn}

Our analogue is as follows:

\begin{defn}
Two circular pairs $(P;Q)$ and $(R;S)$ are \textbf{weakly separated} if $P \cup R$ is weakly separated from $Q \cup S$, and $P \cup S$ is weakly separated from $Q \cup R$.
\end{defn}

\begin{rmk}
Note that $(P;Q)$ is weakly separated from itself and from $(\widetilde{Q};\widetilde{P})$. Furthermore, $(P;Q)$ is weakly separated from $(R;S)$ if and only if $(\widetilde{Q};\widetilde{P})$ is, so under the convention $(P;Q)=(\widetilde{Q};\widetilde{P})$, weak separation is well-defined.
\end{rmk}

\begin{conjecture} \label{wsconj}
Let $C$ be a set of circular minors, for an $n\times n$ generic response matrix.
\begin{itemize}
\item[$\mathsf{P}$:] $C$ is a minimal positivity test.
\item[$\mathsf{S}$:] $C$ is a maximal set of pairwise weakly separated circular pairs.
\item[$\mathsf{C}$:] $C$ is a cluster of $\mc{LM}_{n}$.
\end{itemize}
\end{conjecture}

Conjecture \ref{wsconj} has been computationally verified for $n \leq 6$. We now prove various weak forms of this conjecture. First, for all clusters $C$ of $\mc{LM}_{n}$ that are reachable from the initial seed via Grasmann-Pl\"{u}cker Relations (cf. Conjecture \ref{getall}), the elements of $C$ are pairwise weakly separated:

\begin{proposition}\label{pluckerpreservesws}
If $C$ is a set of pairwise weakly separated circular pairs such that, for some substitution of values into (\ref{P1}) or (\ref{P2}), all the terms on the right hand side, and one term $(P;Q)$ on the left hand side, are in $C$, then the remaining term $(R;S)$ on the left hand side is weakly separated from all of $C-(P;Q)$.
\end{proposition}

\begin{proof}
Let $a,b,c,d$ be as in (\ref{P1}) or (\ref{P2}). It is clear that $(R;S)$ can only be non-weakly separated from an element of $C-(P;Q)$, if $a,b,c,d$ are boundary vertices forcing the non-weak separation.. However, this is easily seen to be impossible.
\end{proof}

When restricting ourselves to clusters of solid minors, the analogue of Corollary \ref{symallconsec} for $\mc{LM}_n$ matches exactly with a weak form of the equivalence $\mathsf{S}\Leftrightarrow\mathsf{C}$ in Conjecture \ref{wsconj}.

\begin{proposition}\label{wsconsec}
A set $C$ of solid circular pairs can be reached (as a cluster) from the initial cluster $\mc{S}_{n}$ in $\mc{LM}_n$ using only mutations of the form (\ref{P1}) if and only if $C$ is a set of $\binom{n}{2}$ pairwise weakly separated solid circular pairs.
\end{proposition}

\begin{proof}
The elements of the initial cluster $\mc{S}_{n}$ in $\mc{LM}_n$, which consists of the diametric pairs $\mc{D}_{n}$, are easily seen to be pairwise weakly separated. Then, by Proposition \ref{pluckerpreservesws}, any cluster we can reach from $\mc{S}_{n}$ using only mutations of the form (\ref{P1}) must also be pairwise weakly separated.

Conversely, consider any set $C$ of $\binom{n}{2}$ pairwise weakly separated solid circular pairs. Let $C' = \{(P;Q)' \mid (P;Q) \in C\}$, and notice that $|C'| = 2 \binom{n}{2}$. By Corollary \ref{symallconsec} and Lemma \ref{isom}, it is enough to prove that that $C' \cup \{\epair\}$ is a solid cluster (see Definition \ref{solidcluster}) in $\mc{CM}_{n}$. From here it will follow by definition that $C'$ is a symmetric solid seed, meaning $C$ can be reached from $\mc{S}_{n}$ in $\mc{LM}_n$ using only mutations of the form (\ref{P1}), as desired.

Before proceeding, it is straightforward to check that circular pairs $(P;Q) = (p_1, \ldots, p_a; q_1, \ldots, q_a)$ and $(R;S) = (r_1, \ldots, r_b; s_1, \ldots, s_b)$ are weakly separated if and only if the following four intersections are non-empty:
\begin{equation*}
\{p_1, q_1\} \cap (R \cup S),~\{p_a, q_a\} \cap (R \cup S),~\{r_1, s_1\} \cap (P \cup Q),~\{r_b, s_b\} \cap (P \cup Q).
\end{equation*}

We now prove that $C'\cup\{\epair\}$ is a solid cluster. First, notice that if non-symmetric circular pairs $(P;Q)' = (p_1, \ldots, p_a; q_1, \ldots, q_a)'$ and $(R;S)' = (r_1, \ldots, r_b; s_1, \ldots, s_b)'$ are such that $k(P;Q)' = k(R;S)'$ and $T(P;Q)' = T(R;S)'$, but $D(P;Q)' \neq D(R;S)'$, then $(P;Q)$ and $(R;Q)$ are not weakly separated. Hence, at most one of $(P;Q)'$ and $(R;S)'$ is in $C'$. As there are exactly $2\binom{n}{2}$ choices of $T$ and $k$ that give valid non-symmetric solid circular pairs, there must be one element of $C'$ corresponding to each choice of $(T,k)$.

Second, consider any adjacent $(P;Q)'$ and $(R;S)'$ in $C'$. Without loss of generality, one of
\begin{itemize}
\item $k(P;Q)' = k(R;S)'$ and $T(P;Q)' = T(R;S)' + \frac12$,
\item $T(P;Q)' = T(R;S)'$ and $k(P;Q)' = k(R;S)' + 1$.
\end{itemize}
holds. In either case, because $(P;Q)'$ and $(R;S)'$ are weakly separated, we can see that $|D(P;Q)' - D(R;S)'| = 2$. It follows that $C' \cup \{\epair\}$ is a solid cluster, so we are done.
\end{proof}

We now relate $\mathsf{C}$ and $\mathsf{P}$. Recall that, by Lemma \ref{LPpos}, any if $C$ satisfies $\mathsf{C}$, then $C$ is a positivity test. Furthermore, $|C| = \binom{n}{2}$. We can prove, similarly to \cite[Theorem 1.2]{leclerc}, that:

\begin{proposition} \label{bound}
If $C$ satisfies $\mathsf{S}$, then $|C| \leq \binom{n}{2}$.
\end{proposition}

In fact, we can prove a slightly stronger result by interpreting a circular pair as a set of edges.

\begin{defn}
For a circular pair $(P;Q) = (p_1, \ldots, p_k; q_1, \ldots, q_k)$, define $E(P;Q) = \{ \{p_i, q_i\} \mid i \in \{1, \ldots, k \}\}$ (cf. Definition \ref{connset}). Similarly, for a set $D \subset \{ \{i,j\} \mid 1 \leq i < j \leq n \}$ of edges such that no two edges in $D$ cross, let $P(E)$ be the circular pair for which $E(P(D)) = D$.
\end{defn}

\begin{proposition} \label{strongbound}
If $C$ is a set of pairwise weakly separated circular pairs with 
\begin{equation*}
E = \bigcup_{(P;Q) \in C} E(P;Q),
\end{equation*}
then $|C| \leq |E|$.
\end{proposition}

\begin{proof}
Procced by induction on $|E|$. The case $|E|=0$ is trivial, so assume the result is true for $|E| < m$. Suppose that we have $C,E$ with $|E|=m$, and assume for sake of contradiction that $|C|>m$. Choose some $\{a,b\} \in E$ such that, for any other $\{c,d\} \in E$, $c$ and $d$ do not both lie on the arc drawn from $a$ to $b$ in the clockwise direction (this arc is taken to include both $a$ and $b$). Now, letting $E' = E \setminus \{ \{a,b\} \}$, define the projection map $J:2^E \to 2^{E'}$ by:
\begin{equation*}
J(D) = \begin{cases}
D \setminus \{\{a,b\}\}  & \text{if } \{a,b\} \in D, \\
D &\text{otherwise.}
\end{cases}
\end{equation*}
We may define $J$ for circular pairs analogously: $J(P;Q) = V(J(E(P;Q)))$, and let $C' = \{ J(P;Q) \mid (P;Q) \in C \}$. Let us now prove two lemmas.

\begin{lemma}\label{wsinJ}
The elements of $C'$ are pairwise weakly separated.
\end{lemma}
\begin{proof}
Assume, for sake of contradiction, that we have $(P;Q), (R;S) \in C$, such that $J(P;Q)$ and $J(R;S)$ are not weakly separated. If $\{a,b\}\in E(P;Q),E(R;S)$ or $\{a,b\}\notin E(P;Q),E(R;S)$, the claim is clear. Thus, we may assume without loss of generality, that $a \in P$, $b \in Q$, and $\{a,b\} \notin E(R;S)$. Because $J(P;Q)$ and $J(R;S)$ are not weakly separated, suppose, without loss of generality, that $w,y \in R \cup (P \setminus \{a\})$ and $x,z \in S \cup (Q \setminus \{b\})$ such that $w,x,y,z$ are in clockwise order, and furthermore $w,y \notin S \cup (Q \setminus \{b\})$ and $x,z \notin R \cup (P \setminus \{a\})$. Note that, if $a,b \notin \{w,x,y,z\}$, then $w,x,y,z$ would also show that $(P;Q)$ is not weakly separated from $(R;S)$.

Assume that $a=w$; the other cases are similar. First, suppose that $b \neq x$. Then, we must have $w \in R$, and we obtain a similar contradiction to before. On the other hand, if $b=x$, then $a=w \in R$ and $b=x \in S$. But, since $\{a,b\} \notin E(R;S)$, $a$ and $b$ are in distinct non-intersecting edges in $E(R;S)$. Because these edges are also in $E$, we have a contradiction of the construction of $\{a,b\}$. The lemma follows.
\end{proof}

\begin{lemma}\label{atmost1J}
There is at most one $(P;Q) \in C$ with $\{a,b\} \in E(P;Q)$ and $J(P;Q) \in C$.
\end{lemma}

\begin{proof}
Assume for sake of contradiction that we have distinct circular pairs $(P;Q), (R;S) \in C$, with $\{a,b\} \in E(P;Q), E(R;S)$, and $J(P;Q), J(R;S) \in C$. Without loss of generality, assume that $a \in P,R$ and $b \in Q,S$. By the fact that $J(P;Q), J(R;S) \in C$ and the construction of $\{a,b\}$, there exist two points $u,v$ on the clockwise arc from $b$ to $a$ not containing its endpoints, which are both in exactly one of $P,Q,R,S$.

First, if $u \in Q$ and $v \in P$, then the points $a,b,u,v$ force $J(P;Q)$ and $(R;S)$ not to be weakly separated. However, $J(P;Q),(R;S)\in C$, so we have a contradiction. 

If $u \in Q$ and $v \in R$, then we get a similar contradiction if $d(b,v)<d(b,u)$, so we have that $a,b,u,v$ are in clockwise order. Because $|Q|=|P|$ and $|R|=|S|$, there must be an $x$ such that $x \in P \cup S$ but $x \notin R \cup Q$.

We have four cases for the position of $x$, relative to $a,b,u,v$. If $a,x,b$ are in clockwise order, then we get a contradiction of our construction of $\{a,b\}$. If $b,x,u$ are in clockwise order, then $a,b,x,u$ either contradicts that $(P;Q)$ is a circular pair or that $(P;Q)$ is weakly separated from $J(R;S)$. The case in which $v,x,a$ are in clockwise order is similar. Finally, if $u,x,v$ are in clockwise order, then either $a,b,u,x$ or $a,b,x,v$ contradicts that either $(P;Q)$ is weakly separated from $J(R;S)$ or that $J(P;Q)$ is weakly separated from $(R;S)$.

The other cases follow similarly.
\end{proof}

We can now finish the proof of Proposition \ref{strongbound}. By Lemma \ref{wsinJ}, the elements of $C'$ are pairwise weakly separated, and we also have $E'=\bigcup_{(P;Q)\in C'}E(P;Q)$. Thus, by the inductive hypothesis, $|C'|\le|E'|=|E|-1$. However, it is easy to see from Lemma \ref{atmost1J} that $|C'|\ge|C|-1$, so the induction is complete.
\end{proof}

Now, Proposition \ref{bound} follows easily by taking $E = \{ \{i,j\} \mid 1 \leq i < j \leq n \}$ in Proposition \ref{strongbound}. Proposition \ref{strongbound} also has another natural corollary:

\begin{corollary}
For any set $S$ of pairwise weakly separated circular pairs, there is an injective map $e : S \to \{ \{i,j\} \mid 1 \leq i < j \leq n \}$ such that $e(P;Q) \in E(P;Q)$ for each $(P;Q) \in S$.
\end{corollary}

\begin{proof}
Proposition \ref{strongbound} gives exactly the condition required to apply Hall's marriage theorem.
\end{proof}

\section{Acknowledgments}

This work was undertaken at the REU (Research Experiences for Undergraduates) program at the University of Minnesota-Twin Cities, supported by NSF grants DMS-1067183 and DMS-1148634. The authors thank Joel Lewis, Gregg Musiker, Pavlo Pylyavskyy, and Dennis Stanton for their leadership of the program, and are especially grateful to Joel Lewis and Pavlo Pylyavskyy for introducing them to this problem and for their invaluable insight and encouragement. The authors also thank to Thomas McConville for many helpful discussions. Finally, the authors thank Vic Reiner, Jonathan Schneider, and Dennis Stanton for suggesting references, and Damien Jiang and Ben Zinberg for formatting suggestions.

\appendix

\section{Proof of Theorem \ref{circposthm} in the BSP case}\label{BSPcase}

\numberwithin{equation}{section}

We now finish the proof of Theorem \ref{circposthm}. Recall that $S_{0}$ is an electrical positroid for which no circular planar graph $G$ has $\pi(G)=S_{0}$, and that $S_{0}$ is chosen to be maximal among electrical positroids with this property. By assumption, $S_{0}$ has the $(i,i+1)$-BEP for each $i$, and does not have the $1$-BSP. Furthermore, recall the construction of $S_{1}$ from the end of \S\ref{posproof}. Then, we have:

\begin{lemma}
$S_{1}$ is an electrical positroid, and has the $1$-BSP.
\end{lemma}
\begin{proof}
Straightforward.
\end{proof}

By assumption, $S_{1}=\pi(G_{1})$, for some circular planar graph $G_{1}$, which has a boundary spike at 1. Let $G_{0}$ be graph obtained after contracting the boundary spike in $G_{1}$. We will prove that $\pi(G_{0})=S_{0}$, which will yield the desired contradiction. The proof is similar to the case handled in \S\ref{posproof}.

Recall the notation from Definition \ref{betternotation}, where we let $A_{k,\ell}$ denote the sequence $a_{k},\ldots,a_{\ell}$.

\begin{defn}
A circular pair $(P;Q)=(A_{1,n};B_{1,n})$ is said to be \textbf{incomplete} if $(P;Q)\notin S$ but $(P;Q')=(A_{1,n}; 1,B_{2,n})\in S$ and $(P';Q)=(1,A_{1,n};B_{1,n})\in S$. If, on the other hand, $(P;Q)\in S$ in addition to $(P;Q')$ and $(P';Q)$, $(P;Q)$ is said to be \textbf{complete}.
\end{defn}

We also define the set $\mc{P}$ of \textbf{primary} circular pairs as in \S\ref{posproof}, where we take circular pairs of the form $(P;Q)=(A_{1,n};B_{1,n})$ with the property that $(A_{1,n}; 1, B_{2,n}),(1,A_{2,n},B_{2,n})\in S$. It is easy to see that the analogue of Lemma \ref{crossing} holds when $(P;Q)$ is incomplete. Then, because $S_{0}$ has all BEPs, the primary circular pairs $(A_{1,n};B_{1,n})$ will all have $a_1=2,b_1=n$. We now prove a series of lemmas, mirroring those in \ref{posproof}.

\begin{lemma}
For an incomplete circular pair $(P;Q)=(A_{1,n},B_{1,n})$, any electrical positroid $Z$ satisfying $S_{0}\cup\{(P;Q)\}\subset Z\subset S_{1}$ con  $(P+a;Q+b)$ with $a>a_{|P|},b<b_{|P|}$ when $(P+a;Q+b)$ is incomplete.
\end{lemma}
\begin{proof}
First by Axiom \ref{axside} in $Z$, $(P;Q)\in Z$ and $(P+a-a_1;Q+b-b_1)\in Z$ implies that we either have our claim, or we have $(P+a-a_1;Q)\in Z$ and $(P;Q+b-b_1)\in Z$. We first apply Axiom \ref{axmid} to $1,a_1,a;b_1$ on the circular pair $(P+a+1,Q+b)$. We have $(P+1+a-a_1;Q+b)\in S$ by definition, and we also have $(P;Q+b-b_1)\in Z$, so this implies that we either have our claim, or we have $(P+1;Q+b)\in Z$. The latter then implies that $(P;Q+b-b_1)\in S$. A similar argument for Axiom \ref{axmid} on $a_1;1,b_1,b$ gives either our claim or that $(P+a-a_1;Q)\in S$. Then, Axiom \ref{axmid} implies that $(P;Q)\in S$, a contradiction. Thus, we have our claim.
\end{proof}

\begin{lemma}\label{PCPsp}
For any incomplete circular pair $(P;Q)$, there exists a circular pair $(P';Q')\in \mc{P}$ such that any electrical positroid $Z$ satisfying $S_{0}\cup\{(P;Q)\}\subset Z\subset S_{1}$ contains $(P;Q)$.
\end{lemma}

\begin{proof}
Proceed by induction on $i$, where $i$ is such that the first $i$ connections (see Definition \ref{connset}) of $(P;Q)$ are the same as those of a primary circular pair. The base case may be handled similarly as in the proof of Lemma \ref{PCPcont}. Now, let $(R;T)=(A_{1,n};B_{1,n})$ be the primary circular pair such that if $(P;Q)=(C_{1,m};D_{1,m})$, then $a_i\le c_i, b_i\le d_i$ for all $i$; $(R;T)$ exists by an identical argument as in the proof of Lemma \ref{PCPcont}. Call $(R;T)$ the primary circular pair associated to $(P;Q)$. 

Recalling that $a_1=2,b_1=n$, we first need to show that $(2,A_{2,i+1}, C_{i+2,m}; 1, B_{2,i+1}, D_{i+2,m})\in S$. The same result replacing $(2,1)$ with $(1,n)$ would follow from an identical argument. By the definition of the $(1,2)$-BEP, we need to show that $(A_{2,i+1}, C_{i+2,m}; B_{2,i+1}, D_{i+2,m})\in S$. In the case that $i>2$, we do so by applying Lemmas \ref{crossing} and Lemma \ref{rtol}. In the case that $i=0$, we may apply the Subset Axiom. Finally, in the case that $i=1$, we may apply Lemma \ref{walk}.

We now claim that, if $(A;B)=(A_{1,i+1}, C_{i+2,m}; B_{1,i+1}, D_{i+2,m})\in Z$, then $(P';Q')\in Z$. The lemma will then follow, because $(A;B)$ and $(P;Q)$ are easily seen to have the same primary associated circular pair. If $i>0$, the proof of the claim is identical to that of \ref{PCPcont}, so assume that $i=0$. In this case, we have $(a_1,C_{2,m};b_1,D_{2,m}),(C_{1,m};1,D_{2,m})\in Z$. If $a_1\neq c_1$, Axiom \ref{axcross} implies that we have $(C_{1,m};b_1,D_{2,m})\in Z$. Then, as $(1,C_{2,m};D_{1,m})\in Z$, we are done by Axiom \ref{axcross}.
\end{proof}

\begin{lemma}\label{primaryBSP}
There is exactly one circular pair in $\mc{P}$ that does not lie in $S_{0}$, which we call the \textbf{$S_{0}$-primary circular pair.}
\end{lemma}
\begin{proof}
The argument is the same to that of Lemma \ref{unPCP}.
\end{proof}

\begin{lemma}\label{spPCP}
For any incomplete circular pair $(P;Q)$, any electrical positroid $Z$ satisfying $S_{0}\cup\{(P;Q)\}\subset Z\subset S_{1}$ contains the $S_{0}$-primary circular pair.
\end{lemma}
\begin{proof}
Proceed by retrograde induction on $i$, where $i$ is such that the first $i$ connections of $(P;Q)$ are the same as those of the $S_{0}$-primary circular pair. If $i>0$, we can argue exactly as in the proof of Lemma \ref{contPCP}. Thus, assume that $i=0$.

Let $(R;T)=(A_{1,n}; B_{1,n})$ be the $S_{0}$-primary circular pair. By Lemma \ref{primaryBSP}, $(P;Q)=(C_{1,m}; D_{1,m})$ has the property that $a_i\le c_i, b_i\le d_i$ for all $i$. Also, by how the construction of $S_{1}$, for any circular pair $(C,D)$, $(C+1;D+n)\in S_{0} \Leftrightarrow (C+1;D+n)\in S_{1}$ and $(C+2;D+1)\in S_{0} \Leftrightarrow (C+2;D+1)\in S_{1}$. Therefore, it follows that $(P+1;Q+n),(P+2;Q+1)\notin Z$. Then, by Axiom \ref{axside}, $(2, C_{2,m}; D_{1,m}\in Z$, and another application of Axiom \ref{axside} yields $(2, C_{2,m}; n,D_{2,m})\in Z$, completing the proof.
\end{proof}

\begin{lemma}\label{uniclossp}
For any two incomplete circular pairs $(P;Q)$ and $(P';Q')$ any electrical positroid $Z$ containing $S$ and contained in $S'$ with $(P;Q)$ must contain $(P';Q')$.
\end{lemma}
\begin{proof}
By Lemma \ref{PCPsp},  $Z$ contains the primary circular pair. The claim then follows by Lemma \ref{spPCP}.
\end{proof}

\begin{lemma} Let $T=S_{0}\cap S'_{0}$. Then, $T$ is an electrical positroid.
\end{lemma}
\begin{proof}
The proof follows the same outline as that of Lemma \ref{findclos}; here, we verify that $T$ satisfies each electrical positroid axiom. By construction, $S_{0}$ and $S'_{0}$ only differ in the circular pairs $(P;Q)$ for which $(P-a_1+1;Q),(P;Q-b_1+1)\in S_{0}\cap S'_{0}$. In particular, $1\neq P,Q$. $T$ is easily seen to satisfy the electrical positroid axioms other than \ref{axside} and \ref{axmid}. 

We first consider Axiom \ref{axside}: suppose that $(P-a;Q-c),(P-b;Q-d)\in T$; we show that either $(P-a;Q-d),(P-b;Q-c)\in T$ or $(P;Q),(P-a-b;Q-c-d)\in T$. We have the following cases:

\begin{itemize}
\item
$a,c\neq 1$. Suppose that $(P-a;Q-d)\in S'_{0}$. Then, either $(P-a;Q-d)\in S_{0}$ or $(P-a;Q-b_1-d+1)\in S_{0}$. Axiom \ref{axcross} applied to $(P;Q-d+1)$ with $a,b,1,b_1$ gives $(P-a;Q-d)\in S_{0}$. Similarly, the roles of $S'_{0}$ and $S_{0}$ may be swapped, and we may apply the same argument with $(P-b;Q-c)$.  Thus, either $S_{0}$ and $S'_{0}$ both contain $(P-a;Q-d)$ and $(P-b;Q-c)$ or both do not, in which case they both contain $(P;Q)$ and $(P-a-b;Q-c-d)$.

\item
$a=1, b\neq a_2$. Suppose that $(P-a;Q-d)\in S'_{0}$ and $(P-b;Q-c)\in S'_{0}$. Then, $(P-b;Q-c)\in S_{0}$. If $(P-a;Q-d)\in S_{0}$, we are done. Otherwise, if $(P-a;Q-d)\notin S_{0}$, as $S_{0}$ is an electrical positroid, we find $(P;Q)\in S_{0}$ and $(P-a-b;Q-c-d)\in S_{0}$. Then, we must have $(P;Q)\in S'_{0}$, in which case we are done, or either $(P-1-b;Q-c-d)\in S'_{0}$ or $(P-a_2-b;Q-c-d)\in S'_{0}$. In the latter case, Axiom \ref{axleft} applied to $(p;Q-d)$ with $1,a_2,b,c$ gives that $(P-1-b;Q-c-d)\in S'_{0}$. Thus, $S_{0}$ and $S'_{0}$ contain $(P-a-b;Q-c-d)$ and $(P;Q)$, so we are done in this case as well.

\item
The cases $a=1, b=a_2,c\neq b_1$ and $a=1,b=a_2,c=b_1$ are handled by similar logic; the details are omitted. The case $c=1$ is symmetric with $a=1$.
\end{itemize}

We now consider Axiom \ref{axmid}: suppose that $(P-b;Q),(P-a-c;Q-d)\in T$; we show that either $(P-a;Q),(P-b-c;Q-d)\in S$ or $(P-c;Q),(P-a-b;Q-d)\in T$. We have the following cases:

\begin{itemize}
\item
$a,d\neq 1, d\neq b_1$. If $(P-a;Q)\in S'_{0}$, then either $(P-a;Q)\in S_{0}$ or $(P-a;Q-b_1+1)\in S_{0}$. Then, an application of Axiom \ref{axcross} to $(P;Q+1)$ with $a,b,1,b_1$ yields that $(P-a;Q)\in S_{0}$. The same argument holds if $(P-a;Q)\in S_{0}$ to show that $(P-a;Q)\in S'_{0}$ does as well. Now, suppose $(P-b-c;Q-d)\in S'{0}$. Then, either $(P-b-c;Q-d)\in S_{0}$, which case we are done, or $(P-a_1-b-c+1;Q-d)\in S_{0}$. Then, the Subset Axiom, $(P-a_1-b-c;Q-b_1-d)\in S{0}$. Applying Axiom \ref{axsep} to $(P-b;Q)$ with $(a_1,c,b_1,d)$ then yields that $(P-b-c;Q-d)\in S_{0}$, as desired. The same argument holds if we swap the roles of $S_{0}$ and $S'_{0}$.

\item
$a\neq 1, d= b_1$. If $(P-a;Q)\in S'_{0}$, by the same argument as with the case $a,d\neq1,d\neq b_{1}$, we have $(P-a;Q)\in S_{0}$. If $(P-b-c;Q-d)\in S'_{0}$ as well, then either $(P-b-c;Q-d)\in S_{0}$, in which case we are done, or $(P+1-a_1-b-c;Q-d)\in S_{0}$. In the latter case, because $S_{0}$ is an electrical positroid, $(P-c;Q)\in S_{0}$. Then, applying Axiom \ref{axleft} to $(P+1-c;Q)$ with $(1,a_1,b,d)$ gives that $(P-b-c;Q-d)\in S_{0}$, so we are done.

\item
The cases $a=1$ and $d=1$ are handled with similar logic; the details are omitted.
\end{itemize}
We have exhausted all cases, so the proof of the lemma is complete.
\end{proof}

Now, by Lemma \ref{uniclossp}, we have $S_{0}=S'_{0}$, so the proof of Theorem \ref{circposthm} is complete.

\section{Mutation at limiting minors in $\mc{LM}_{n}$}\label{nonplucker}

\numberwithin{equation}{section}

Recall that, to complete the proof of Theorem \ref{lpalg}, we need an additional technical result, which we state and prove here.

\begin{proposition} \label{mutlimit}
From the initial cluster of $\mc{LM}_n$, mutating at a non-frozen limiting solid circular pair $(P;Q)$ gives a new cluster variable which is a polynomial in the entries of $M$, and relatively prime to $(P;Q)$.
\end{proposition}

\begin{proof}
Consider the limiting solid circular pair $(P;Q)$ of size $k$. Fix the ground set 
\[ (I;J) = \left( \frac{n}{2}, \frac{n}{2}+1,\dots,\frac{n}{2}+k; 2,1,n,n-1,\ldots,n-k+1\right), \]
so that $\Delta$ denotes the determinant of the submatrix of $M$ with rows indexed by $I$ and columns indexed by $J$. Furthermore, let
\[ b=\frac{n}{2}+k-1,~ c = \frac{n}{2}+k,~ d=2,~ e=1,~ f=n-k+2,~ g=n-k+1.  \]
Then, the cluster variable associated to the vertex $(P;Q)$ is $\Delta^{c,dg}$, and its corresponding exchange polynomial in the initial seed of $\mc{LM}_{n}$ is
\[ \Delta^{\emptyset,d} \cdot \Delta^{c,fg} \cdot \Delta^{bc,deg} + \Delta^{\emptyset,g} \cdot \Delta^{c,de} \cdot \Delta^{bc,dfg}.\]
The new cluster variable from mutating at $(P;Q)$ is
\[ \frac{ \Delta^{\emptyset,d} \cdot \Delta^{c,fg} \cdot \Delta^{bc,deg} + \Delta^{\emptyset,g} \cdot \Delta^{c,de} \cdot \Delta^{bc,dfg} }{\Delta^{c,dg}} = \Delta^{b,de} \cdot \Delta^{c,fg} - \Delta^{b,fg} \cdot \Delta^{c,de},\]
where the last equality may be checked directly. We wish to show that $\Delta^{b,de} \cdot \Delta^{c,fg} - \Delta^{b,fg} \cdot \Delta^{c,de}$ is relatively prime to $\Delta^{c,dg}$, which is irreducible, so it is enough to check that $\Delta^{c,dg}$ does not divide $\Pi=\Delta^{b,de} \cdot \Delta^{c,fg} - \Delta^{b,fg} \cdot \Delta^{c,de}$.

Let us outline the argument. If it is the case that $\Delta^{c,dg}$ divides $\Pi$, then each term in the expansion of $\Pi$ must be divisible by a monomial in the expansion of $\Delta^{c,dg}$. However, we claim that this cannot be true. Indeed, any monomial in the expansion of $\Delta^{c,dg}$ contains exactly one factor of a variable $x_{bz}$ in the row of $M$ corresponding to $b$, and the column of $M$ corresponding to $z\neq2,n-k+1$. However, it is easily checked that there are terms of $\Pi$, after expanding and collecting like terms, containing variables $X_{bz}$ with $z=2$, so the claim is established. 
\end{proof}

\end{document}